\documentclass[12pt]{article}

\usepackage{endnotes}
\let\footnote=\endnote

\usepackage{setspace}
\usepackage{natbib}
\bibpunct[, ]{(}{)}{,}{a}{}{,}%
\def\bibfont{\small}%
\def\bibsep{\smallskipamount}%
\def\bibhang{24pt}%
\usepackage{amsthm}

\usepackage[margin=0.7in]{geometry}
\usepackage{graphicx} 
\usepackage{epstopdf}

\usepackage{booktabs} 
\usepackage{array} 
\usepackage{paralist} 
\usepackage{verbatim} 

\usepackage{caption}
\usepackage{subcaption}

\usepackage{amsmath, amssymb}

\usepackage{multirow}
\usepackage{algorithm}
\usepackage{algorithmic}
\usepackage{mathtools}

\newtheorem{prop}{Proposition}[section]
\newtheorem{thm}{Theorem}[section]

\newtheorem{cor}{Corollary}[section]
\newtheorem{dfn}{Definition}[section]
\newtheorem{clm}{Claim}[section]

\usepackage{rotating}
\usepackage{pdflscape}
\usepackage{url}
\graphicspath{{figures/}} 

\newcommand{\exclude}[1]{}
\newcommand{\cB}{\mathcal B}
\newcommand{\cC}{\mathcal C}
\newcommand{\cF}{\mathcal F}
\newcommand{\cG}{\mathcal G}
\newcommand{\cL}{\mathcal L}
\newcommand{\cS}{\mathcal S}

\usepackage{color}
\def\bred{\color{black}}

\begin{document}



\title{A Cycle-Based Formulation and Valid Inequalities for DC Power Transmission Problems with Switching}


\author{Burak Kocuk,  Hyemin Jeon, Santanu S. Dey, \\ Jeff Linderoth, James Luedtke, Andy Sun}
	

\maketitle

\abstract{%
It is well-known that optimizing network topology by switching on and
off transmission lines improves the efficiency of power delivery in
electrical networks.  In fact, the USA Energy Policy Act of 2005
(Section 1223) states that the U.S. should ``encourage, as
appropriate, the deployment of advanced transmission technologies''
including ``optimized transmission line configurations.''  As such,
many authors have studied the problem of determining an optimal set of
transmission lines to switch off to minimize the cost of meeting a
given power demand under the direct current (DC) model of power flow.
This problem is known in the literature as the {\it Direct-Current
Optimal Transmission Switching} Problem (DC-OTS).  Most research on
DC-OTS has focused on heuristic algorithms for generating quality
solutions or on the application of DC-OTS to crucial operational and
strategic problems such as contingency correction, real-time dispatch,
and transmission expansion.  The mathematical theory of the DC-OTS
problem is less well-developed.  In this work, we formally establish
that DC-OTS is NP-Hard, even if the power network is a series-parallel
graph with at most one load/demand pair.  Inspired by Kirchoff's
Voltage Law, we give a cycle-based formulation for DC-OTS, and we use
the new formulation to build a cycle-induced relaxation.  We
characterize the convex hull of the cycle-induced relaxation, and the
characterization provides strong valid inequalities that can be used
in a cutting-plane approach to solve the DC-OTS.  We give details of a
practical implementation, and we show promising computational results
on standard benchmark instances.
}%



\section{Introduction}

An electric power grid is a complex engineered system whose control
and operation is driven by fundamental laws of physics.  The {\it
  optimal power flow} (OPF) problem is to determine a minimum cost
delivery of a given demand for power subject to the power flow
constraints implied by the network.  The standard mathematical model
of power flow uses alternating current (AC) power flow equations.  The
AC power flow equations are nonlinear and nonconvex, which has
prompted the development of a linear approximation known as the direct
current (DC) power flow equations.  DC power flow equations are widely
used in the current industry practice, and the DC OPF problem is a
building block of power systems operations planning.

One consequence of the underlying physical laws of electric power flow
is a type of ``Braess' Paradox'' where removing lines from a
transmission network may result in {\em improved} network efficiency.
\cite{oneill.et.al:05} propose exploiting this
well-known attribute of power transmission networks by switching off
lines in order to reduce generation costs.  
\cite{Fisher} formalized this notion into a mathematical optimization
problem known as the Optimal Transmission Switching (OTS) problem.  The OTS problem is the OPF problem augmented with
the additional flexibility of changing the network topology by removing transmission lines. While motivated in
\cite{oneill.et.al:05} by an operational problem in which lines may be switched off to improve efficiency, the same
mathematical switching structure appears also in longer-term transmission network expansion planning problems.

Because of the mathematical complexity induced by the AC power flow
equations, nearly all studies to date on the OTS problem have used the
DC approximation to power flow (e.g. see \cite{Fisher,Barrows12,Fuller12,Wu13}).
With this approximation, a mixed-integer linear programming (MILP)
model for DC-OTS can be created and input to existing MILP software.
In the MILP model, binary variables are used to model the changing
topology and enforce the network power flow constraints on a line if
and only if the line is present.  Models with many ``indicator
constraints'' of this form are often intractable for modern
computational integer programming software, since the linear
relaxations of the formulations are typically very weak.  Previous
authors have found this to be true for DC-OTS, and many heuristic
methods have been developed based on ranking lines (\cite{Barrows12,Fuller12,Wu13}), or by imposing an (artificial)
cardinality constraint on the number of lines that may be switched off
in a solution (\cite{Fisher}).  The solutions found from these
heuristics have demonstrated that significant efficiency
gains are possible via transmission switching.

The paper by \cite{coffrin.et.al:14} offers a criticism of the use
of the DC-approximation to model power flow for the OTS problem.  The
paper demonstrates that a direct application of the DC power flow
equations may {\it not} be accurate enough to recover useful AC
operation solutions in the context of the OTS.  They thus argue for
the use of the AC-power flow equations in the OTS problem.  The
authors employ a convex quadratic relaxation of the AC power flow
equations proposed in \cite{hijazi.coffrin.vanhentenryck:13} and embed the
relaxation in a branch-and-bound method to solve the AC-OTS problem.
The recent papers by \cite{barrow.blumsack.hines:14} and \cite{Soroush14} also develop
methods and heuristics for AC-OTS.

\exclude{
Transmission switch with the AC power flow model has also been studied. In \cite{Soroush14}, authors propose a new ranking heuristics based on the AC OPF solution and the corresponding dual variables. 
}

Despite the criticism of the DC power flow model for optimal
transmission switching, there are many planning problems where the DC
power flow model may be an acceptable approximation when combined with network topology
changes.  A survey in \cite{Hedman11} enumerates applications of
transmission switching to improve voltage profiles, reduce congestion
and losses in the system, and to increase reliability of the power
grid.  A standard reliability criterion for the power grid is that the
system must be able to withstand an ``$N-1$'' event---in an
interconnection network with $N$ elements, the system will operate
reliably following the failure of any one of them.  
\cite{Hedman09} augment the DC-OTS model to ensure that the
$N-1$ reliability criteria is satisfied.  Heuristics are 
used to
iteratively decide which lines to be switched off, while preserving
$N-1$ reliability for standard test cases.  The authors show that
significant cost savings from transmission switching is still possible
even if the $N-1$ contingency is required.  

Transmission switching is also an important subproblem in power grid capacity
expansion planning.  In \cite{Khodaei10}, a large MILP model is
constructed to solve an expansion planning problem with contingencies,
the DC power flow equations are used, 
and transmission switching is allowed in finding the best
configuration.  \exclude{The model is solved using Benders'
  decomposition where the master problem deals with the expansion
  decision, while two subproblems are used to generate feasibility and
  optimality cuts based on the contingency requirements and optimality
  conditions.} 
\cite{villumsen.philpott:12} develop a slightly
different model for expansion planning that also relies on the
DC-power flow approximation and transmission switching.  They employ
their method on a case study for a real power system expansion plan in
Denmark (see \cite{Villumsen13}).  In the context of expansion planning, it
is important to note that  the mathematical
structure of line addition is exactly the same as line removal.

\exclude{
In all studies for transmission expansion planning the DC equations
are used to approximate the power flow, and the mathematiocal
structure of line addition is exactly the same as in optimal
transmission switching}

\exclude{
A case study is provided in \cite{Villumsen13} about the capacity expansion problem in Denmark. The solution of a two-stage stochastic mixed-integer problem demonstrates that transmission switching not only reduces the generation cost in congested networks, but also alters the capacity expansion plan.
}

\exclude{
Switching can be beneficial in terms of cost savings but one should
also consider the reliability side. For instance, it may happen that
when some of the lines are switched off, the remaining network may
become disconnected. This undesired situation can be prevented by
using  ``connectivity-ensuring constraints" as discussed in
\cite{Ostrowski14}. Using graph-theoretical results, they add valid
inequalities for each biconnected component and use a special
branching procedure to ensure connectivity.
}

Authors who have combined the DC-approximation to power flow with
the flexibility of modifying the network topology have found that the
resulting MILP programming model is very challenging and called for a
more systematic study on its underlying mathematical structure.   For
example,  the authors of \cite{Hedman10} state that 
\begin{quote}
``When solving the transmission switching problem, $\ldots$ the techniques for
closing the optimality gap, specifically improving the lower bound, are
largely ineffective.''
\end{quote}
A primary focus of our work is an attempt to change this reality by
developing strong classes of valid inequalities that may be applied to
power systems planning problems that involve the addition or removal
of transmission elements.  Our paper makes the following
contributions:
\begin{itemize}
\item Using a formulation based on Kirchoff's Voltage Law, we give a
cycle-based linear mixed integer programming (MIP) formulation for DC-OTS.
\item We formally establish that the DC-OTS problem is NP-Hard, even
  if the interconnection network is a series-parallel graph and there
  is only one generation/demand pair.  (Additional complexity results
  were recently independently established by
  \cite{lehmann.grastien.vanhentenryck:14}).  
\item Using the cycle-based formulation as inspiration, we derive
  classes of strong valid inequalities for a cycle--relaxation of
  the DC-OTS.  We additionally establish that inequalities define the
  convex hull of the cycle-relaxation, and we show how to separate the
  inequalities over a given cycle in polynomial time.
\item We perform computational experiments focusing on improving
  performance of integer programming-based methods for the OTS using
  the DC power flow approximation.   We show that the new 
  inequalities can help improve solution performance of commercial MIP
  software. 
\end{itemize}
{\bred Although our valid inequalities are new, they add to the wide body of literature that has made productive use of cycles in a graph to derive valid
inequalities for discrete optimization problems; see e.g., \citep{padberg:mp73,cutpolytope:86,ferreiraetal:96,cmix:05} for just a few
examples.}

\section{DC OPF and OTS Formulations} 
\label{sec:DCOPF}

A power network consists of a set of buses $\cB$, transmission lines $\cL$, and generators $\cG \subseteq \cB$.  We assume that the graph $G=(\cB,\cL)$ is connected.  Each line $(i,j) \in \cL$ is given an (arbitrary) orientation, with the convention that power flow in the direction from $i \rightarrow j$ is positive, while power that flows along line $(i,j)$ in the direction $j \rightarrow i$ is negative.   Each bus $i \in \cB$ has a set of adjacent buses $\delta(i) \subseteq \cB$.  We use the standard notation that $\delta^+(i) := \{j \in \cB : (i,j) \in \cL\}$ and $\delta^-(i) := \{j \in \cB : (j,i) \in \cL\}$.  The required load at each bus is given as $p_i^d$, $i \in \mathcal{B}$. In the DC model, power flow on a transmission line is proportional to the difference in phase angles of voltages at the two ends of the line.  The constant of proportionality is known as the {\it susceptance} of line $(i,j) \in \cL$, which we denote as $B_{ij}$.  Each transmission line $(i,j) \in \cL$ has an upper bound $\bar f_{ij}$ on the allowed power flow.  Finally, the power produced at each generator $i \in \cG$ is constrained to lie in the interval $[p_i^{\text{min}}, p_i^{\text{max}}]$ with an associated unit production cost of $c_i$. 
{\bred We assume that there is at most one generator at each bus.  This assumption is without loss of generality, since the models we present can be extended to the case with multiple generators by replacing the dispatch variable by the sum of different dispatch variables attached to a particular bus.} 
{\bred We also assume that the cost is a linear function of the production quantities at the generators.  This
assumption, while standard in the literature surrounding the DC-OTS \citep{Fisher,Fuller12}, ignores the convex
quadratic portion of generation costs. As we discuss in Section \ref{sec:valid}, our formulation and valid
inequalities can still be applied when the
cost is a nonlinear function.} 

\subsection{Angle Formulation of DC-OPF}
The standard formulation of DC OPF has three classes of decision variables.  The variable $p_i^g$ is the power  generation at generator $i \in \cG$, and the variable $\theta_i$ is the voltage angle at bus $i \in \cB$.  The variable $f_{ij}$ represents the power flow along line $(i,j) \in \cL$.   With these decision variables, we can write a linear programming problem to minimize the generation cost of meeting power demands $p_i^d$ as follows:
\begin{subequations} \label{eq:dc-opf}
\begin{align}
  \min  &\hspace{0.5em}  \sum_{i \in \mathcal{G}} c_i p_i^g  \label{objective} \\
  \mathrm{s.t.}   &\hspace{0.5em} p_i^g-p_i^d = \sum_{j \in \delta^+(i)} f_{ij} - \sum_{j \in \delta^-(i)} f_{ji}   & i& \in \mathcal{B} \label{activeAtBus}\\
  & \hspace{0.5em} B_{ij}(\theta_i-\theta_j) = f_{ij} &(&i,j) \in \cL \label{eq:def-f}\\
  & \hspace{0.5em}   -\bar f_{ij} \le f_{ij}   \le \bar f_{ij}  &(&i,j) \in \mathcal{L} \label{powerOnArc} \\
  & \hspace{0.5em} p_i^{\text{min}}  \le p_i^g \le p_i^{\text{max}}    & i& \in \mathcal{G}. \label{activeAtGenerator}
\end{align}
\end{subequations}
By substituting the definition of power flow from \eqref{eq:def-f} into equations~\eqref{activeAtBus} and~\eqref{powerOnArc}, the $f$ variables may be projected out of the formulation. 

\exclude{
Here, we should point out  that lower and upper bounds on variables $\theta_i$ are put only for convenience. The actual restriction is on the difference of $\theta_i-\theta_j$, which cannot exceed $2\pi$. However, such a constraint is always redundant since the ratio ${ \bar f_{ij}} / {B_{ij}}$ is already much smaller than $2\pi$. 
}

\subsection{Cycle Formulation of DC-OPF}
The DC power flow model gets its name from the fact that the equations describing the power flow in network are the same as those that describe current flow in a standard direct current electric network.  The constraints \eqref{activeAtBus} describe Kirchoff's Current Law (KCL) at each bus, and the equations \eqref{eq:def-f} that define the branch current follow from Ohm's Law.  With this analogy, it is natural to think about the alternative way to represent power flows in a DC circuit---using the branch current $f_{ij}$ and Kirchoff's Voltage Law (KVL).  Kirchoff's Voltage Law states that around any {\it directed} cycle $C$, the voltage differences must sum to zero:
\begin{equation}\label{eq:KVL}
\sum_{ (i,j) \in C} (\theta_i - \theta_j) = \sum_{(i,j) \in C} \frac{f_{ij}}{B_{ij}} = 0.
\end{equation}
If the directed cycle $C$ contains arc $(i,j)$, but the transmission line has been given the alternate orientation ($(j,i) \in \cL$), we adjust \eqref{eq:KVL} by flipping the sign of the susceptance:
\[ \bar{B}_{ij}^C = \left\{ \begin{array}{cl} 
B_{ij} & \mbox{if } (i,j) \in C, (i,j) \in \cL\\
-B_{ij} & \mbox{if } (i,j) \in C, (j,i) \in \cL
\end{array} \right. \]

{\bred
  Our formulation of the DC-OPF relies on the notion of a {\it cycle basis}.
\begin{dfn} \citep{hariharan}
Let $v \in {\{0,\pm1\}}^{|\cL|}$ be an incidence vector for a cycle $C$ in graph $G=(\cB, \cL)$, where
\[ v_{ij}= \left\{ \begin{array}{cl} 
1 & \mbox{if } (i,j) \text{ is traversed in the right direction by } C\\
-1 & \mbox{if } (i,j)\text{ is traversed in the opposite direction by } C \\
0 & \mbox{if } (i,j) \text{ is not in } C .
\end{array} \right. \]
The \textit{cycle space} of $G$ is the vector space that is spanned by the incidence vectors of its cycles. 
A set of cycles is called a \textit{cycle basis} if it forms a basis for this vector space.
\end{dfn}
A cycle basis of $G$ is then a minimal set of cycles of $G$ with the property that all cycles of $G$ are linear combinations of the cycles in the basis.
}

In this equivalent representation, the power flow should satisfy the KVL \eqref{eq:KVL} for each cyle.  Although the number of cycles in a network can be large, it is sufficient to enforce \eqref{eq:KVL} over any set of cycles that forms a cycle basis $\cC_b$ of the network.  if angle differences sum up to zero over the cycle basis, they also must sum up  to zero over any other cycle (e.g. see \cite{bollobas:02}).  Thus, the KVL-inspired formulation for the DC OPF is the following: 

\begin{subequations}
\label{eq:dc-opfkvl}
\begin{align}
  \min  &\hspace{0.5em}  \sum_{i \in \mathcal{G}} c_i p_i^g  \label{objectiveC} \\
  \mathrm{s.t.}  &\hspace{0.5em} {\bred   \eqref{activeAtBus},  \eqref{powerOnArc},  \eqref{activeAtGenerator}  } \\
  & \hspace{0.5em} \sum_{ (i,j) \in C} \frac{f_{ij}}{\bar B^C_{ij}} = 0  & C& \in \mathcal{C}_b \label{cycleBasis}.
\end{align}
\end{subequations}
The voltage angles may be recovered using the equations \eqref{eq:def-f}.

{\bred
\begin{prop}
Formulations \eqref{eq:dc-opf} and \eqref{eq:dc-opfkvl} are equivalent.
\end{prop}
\begin{proof}
In order to prove the equivalence of the formulations, it suffices to show that 
\begin{equation}
  \sum_{ (i,j) \in C} \frac{f_{ij}}{\bar B^C_{ij}} = 0,  \ C \in \mathcal{C}_b\iff\text{there exists } \theta \text{ such that }f_{ij} = B_{ij}(\theta_i-\theta_j),  \ (i,j) \in \cL .
\end{equation}

($\Rightarrow$) First observe that because $\mathcal{C}_b$ is a cycle basis, equations \eqref{cycleBasis} imply  $\sum_{
(i,j) \in C} \frac{f_{ij}}{\bar B^C_{ij}} = 0$ for any cycle $C$.  Now let $\mathcal{T}
= (\mathcal{B}, \mathcal{L'})$ be a spanning tree of $G$. 
Clearly, the following system in $\theta$ has a solution:
\begin{equation}
 B_{ij}(\theta_i-\theta_j) = f_{ij},  \ (i,j) \in \cL' . \label{solve theta}
\end{equation}
Therefore, it suffices to check if a solution of \eqref{solve theta} satisfies $B_{ij}(\theta_i-\theta_j) = f_{ij}$ for $(i,j)\in\cL \setminus \cL'$. Note that for any $(k,l) \in \cL \setminus \cL' $, there exists a unique path  $\mathcal{P}_{kl}$  from $k$ to $l$ in $\mathcal{T}$ and a cycle $C = \mathcal{P}_{kl} \cup \{(l,k)\}$. Then, we have
\begin{equation}
 (\theta_k-\theta_l) =  \sum_{ (i,j) \in \mathcal{P}_{kl}}  (\theta_i-\theta_j) =  \sum_{ (i,j) \in \mathcal{P}_{kl}} \frac{f_{ij}}{\bar B^C_{ij}}  =   \frac{f_{kl}}{\bar B^C_{kl}},
\end{equation}
which implies $B_{kl}(\theta_k-\theta_l) = f_{kl}$. 

($\Leftarrow$) Suppose there exist $\theta$  such that $f_{ij} = B_{ij}(\theta_i-\theta_j)$, for each $(i,j) \in \cL$. Then, for each cycle $C$, we have
\begin{equation}  
\sum_{ (i,j) \in C} \frac{f_{ij}}{\bar B^C_{ij}} =   \sum_{ (i,j) \in C} \frac{ B_{ij}(\theta_i-\theta_j)  }{\bar B^C_{ij}} =   \sum_{ (i,j) \in C} (\theta_i-\theta_j) = 0,
\end{equation} 
which concludes the proof.
\end{proof}
}

\subsection{Angle Formulation of DC-OTS}

The angle-based DC-OPF formulation \eqref{eq:dc-opf} can be easily
adapted to switching by introducing binary variables $x_{ij}$ that
takes the value $1$ if line $(i,j) \in \cL$ is on, and $0$ if the line is
disconnected.  A direct nonlinear formulation of DC-OTS is
\begin{subequations}\label{eq:switch-nonlinear}
\begin{align}
  \min  &\hspace{0.5em}  \sum_{i \in \mathcal{G}} c_i p_i^g \label{objS} \\
  \mathrm{s.t.}  &\hspace{0.5em} {\bred   \eqref{activeAtBus},  \eqref{powerOnArc},  \eqref{activeAtGenerator}  } \\
  & \hspace{0.5em} B_{ij}(\theta_i-\theta_j)x_{ij} = f_{ij}  &(&i,j) \in \cL \label{eq:def-fS}\\
  & \hspace{0.5em}  x_{ij} \in \{0,1\}  &(&i,j) \in \mathcal{L}. \label{binaryS}
\end{align}
\end{subequations}
The constraints \eqref{eq:def-fS} ensure both that Ohm's Law
\eqref{eq:def-f} is enforced if the line is switched on and that power flow
$f_{ij} = 0$ if the line is switched off.  However, these constraints 
\eqref{eq:def-fS} contain nonlinear, nonconvex terms of the form 
$\theta_i x_{ij}$.  The standard way to linearize the inequalities 
{\bred \eqref{eq:def-fS} }  is employed by \cite{Fisher} to produce the
following formulation:
\begin{subequations} \label{eq:Fisher-Form}
\begin{align} 
  \min  &\hspace{0.5em}  \sum_{i \in \mathcal{G}} c_ip_i^g \label{obj} \\
  \mathrm{s.t.}  &\hspace{0.5em} {\bred   \eqref{activeAtBus},    \eqref{activeAtGenerator},    \eqref{binaryS}  } \\
  & \hspace{0.5em}   B_{ij}(\theta_i-\theta_j) -M_{ij}(1-x_{ij})  \le f_{ij} \le B_{ij}(\theta_i-\theta_j) + M_{ij}(1-x_{ij}) &(&i,j) \in \mathcal{L} \label{power big M} \\
  & \hspace{0.5em}   -\bar f_{ij} x_{ij} \le f_{ij}   \le \bar f_{ij} x_{ij}  &(&i,j) \in \mathcal{L} \label{powerLog},
\end{align}
\end{subequations}
where $M_{ij}$ is chosen sufficiently large to make the inequalities \eqref{power
  big M} redundant if $x_{ij} = 0$. 

\subsection{Cycle Formulation of DC-OTS}
Inspired by the cycle formulation \eqref{eq:dc-opfkvl} for the DC-OPF, we
can formulate the DC-OTS problem without angle variables as well. The
full formulation enforces Kirchoff's Voltage Law only if all arcs in a
cycle are switched on.
\begin{subequations}	\label{Full Cycle}
\begin{align}
  \min  &\hspace{0.5em}  \sum_{i \in \mathcal{G}} c_ip_i^g  \label{objectiveCG} \\
  \mathrm{s.t.}  &\hspace{0.5em} {\bred   \eqref{activeAtBus},    \eqref{activeAtGenerator},    \eqref{binaryS}, \eqref{powerLog}  } \\
  & \hspace{0.5em} - M_C  \sum_{ (i,j) \in C}  (1- x_{ij}) \le \sum_{ (i,j) \in C} \frac{f_{ij}}{\bar B^C_{ij}}  \le  M_C  \sum_{ (i,j) \in C}  (1- x_{ij})    & C& \in \mathcal{C} \label{cycleG}.
\end{align}
\end{subequations}
The value $M_C$ must be selected so that the inequalites
\eqref{cycleG} are redundant if $\sum_{(i,j) \in C} (1-x_{ij}) \geq
1$.  In formulation \eqref{Full Cycle}, $\mathcal{C}$ is the set of
{\em all cycles} in the graph $G=(\cB,\cL)$.  The cardinality of $\cC$
is in general quite large, so we do not propose using \eqref{Full
  Cycle} directly.  Rather, we use the formulation \eqref{Full Cycle}
as the starting point for deriving strong valid inequalities in
Section~\ref{sec:valid}. Furthermore, the inequalities \eqref{cycleG} could be added as cuts within a branch-and-cut
algorithm. These inequalities are required to define the feasible region, so the branch-and-cut procedure would search for
a violated inequality from the class \eqref{cycleG} any time it identifies a
candidate solution with the $x$ components binary. (Inequalities added as cuts in this way are sometimes referred to as ``lazy cuts''.)

\section{Complexity of DC-OTS}
\label{sec:complexity}


In this section, we discuss the complexity of the DC optimal transmission
switching problem.  The input to the problem is a power network as
described at the beginning of Section~\ref{sec:DCOPF}.  In the {\em
  feasibility version} of DC-OTS, we ask if there {\em exists} a
subset of lines to switch off such that the DC-OPF is feasible for the
induced topology.  The feasibility version of DC-OTS with a
cardinality constraint has been proven to be NP-Complete
in~\cite{Bienstock} by reduction from the Exact 3-Cover Problem.
Recently, many complexity and approximability results on 
DC-Switching problems were given in 
\cite{lehmann.grastien.vanhentenryck:14}, including the result that
DC-OTS is NP-Hard, even if the underlying graph is a cactus.  
Our results were
established independently, and complement the results of
\cite{lehmann.grastien.vanhentenryck:14} by formally establishing that
the DC-OTS problem is easy if the graph is a tree and NP-Hard even if
there is one generation/load pair on series-parallel graphs.

\begin{prop} \label{connected}
In the DC-OTS, there exists an optimal solution in which the lines switched on form a connected network.
\end{prop}
\begin{proof}
Consider the Cycle Formulation \eqref{Full Cycle} of the DC-OTS, and
let $\mathcal{L}'$ be the active lines in an optimal solution.  Assume
that the network corresponding to this solution has $k$ connected
components.  Since the original network $G(\cB,\cL)$ is connected, we
can find a set of transmission lines $\mathcal{L}''$ with cardinality
$k-1$ such that $G'=(\cB,\mathcal{L}' \cup \mathcal{L}'')$ is
connected.  Now, let $x_{ij}=1$ and $f_{ij}=0$ for all $(i,j) \in
\mathcal{L}''$.  By construction, no new cycles are created by
switching on lines in $\mathcal{L}''$.  Further, the balance
constraints {\bred \eqref{activeAtBus} } and bound
constraints {\bred \eqref{powerLog} } are satisfied.  Hence, we have
demonstrated a new solution with the same objective value where the
network formed by the active lines is connected.
\end{proof}

\begin{cor}
If $G=(\cB,\cL)$ is a tree, the DC-OTS problem is solvable in
polynomial time.
\end{cor}
\begin{proof}
Due to Proposition \ref{connected}, there exists an optimal solution
which induces a connected network. Since removing any line disconnects
the tree, there exists an optimal solution in which all lines are
active. But this is exactly the DC-OPF problem without switching, which
can be solved via linear programming, a problem known to be
polynomially solvable.
\end{proof}

Theorem~\ref{thm:np-hard} establishes that DC-OTS is NP-Complete even
if the power network is a series-parallel graph, and there is only one
demand-supply pair.
\begin{thm}
\label{thm:np-hard}
The feasibility version of DC-OTS is NP-complete even when
$G=(\cB,\cL)$ is a series-parallel graph, there is $|\cG| = 1$
generator, and one node $i \in \cB$ such that $p_i^d \neq 0$.
\end{thm}
\begin{proof}
We prove this result by a reduction from the subset sum problem, which is known to be NP-Complete (\cite{Garey}). Consider an instance of a subset problem as: Given $a_i \in \mathbb{Z}_{++}$ for $i \in \{1, \dots, n\}$ and $b \in \mathbb{Z}_{++}$, does there exist a subset $I \subseteq \{1, \dots, n\}$ such that $\sum_{i \in I} a_i   = b$?  We construct an instance of switching problem as follows:
\begin{enumerate}
\item There are $n + 3$ buses $\{0, 1, \dots, n, n +1, n +2\}$.
\item Following are the lines: $(0, i)$  for all $i\in \{1, \dots, n\}$; $(i, n+1)$ for all $i\in \{1, \dots, n\}$; $(n +1, n+2)$; $(0, n +2)$.
\item The capacities of the lines are: $\frac{a_i}{b}$ for the line $(0,i)$ and $(i, n + 1)$ for all $i\in \{1, \dots, n\}$; $1$ for $(n +1, n+2)$ and $(0, n +2)$.
\item The susceptances of the lines are: ${2a_i}$ for the line $(0,i)$ and $(i, n + 1)$ for all $i\in \{1, \dots, n\}$; $1$ for $(n +1, n +2)$; $ \frac{b}{b+1}$ for $(0, n+2)$.
\item There is a generation of $2$ at bus $0$ and load of $2$ at bus $n +2$.
\end{enumerate}
Clearly, the size of the instance of the switching problem is polynomial in the size of the given instance of the subset sum problem. Also note that the graph is a series parallel graph and there is only one  demand supply pair.

We now verify that the subset sum problem is feasible if and
only if the switching problem is feasible.

$( \Rightarrow  )$: Since the subset sum problem is feasible, let
$\sum_{i \in I}a_i = b$ where $I \subseteq \{1, \dots, n\}$. Then
construct a solution to the switching problem as follows: Switch off
the lines $(0,i), (i, n +1)$ for $i \in \{1, \dots, n\}\setminus I$.
It is straightforward to establish that a feasible solution to the
DC-OTS exists.  (In the solution, the angle at bus $0$ is
$1+\frac{1}{b}$, the angle at bus $i$ is $1+\frac{1}{2b}$ for all $i
\in I$, the angle at bus $n+1$ is $1$, and the angle at bus $n+2$ is 0).

$( \Leftarrow  )$: The subset sum problem is infeasible and assume by
contradiction that the switching problem is feasible. Then note that
the flow in arcs $(0, n+2)$ and $(n +1, n +2)$ are $1$ each (and these
lines are not switched off). WLOG, let the angle at bus $n +2 $ be $0$. This implies that the angle at bus $0$ is
$1+\frac{1}{b}$ and at bus $n +1$ is $1$. Then note that if a pair of lines $(0, i), (i , n +1)$ is not switched off,
this implies that the angle at bus $i$ is
$1+\frac{1}{2b}$ and the resulting flow is $\frac{a_i}{b}$ along the path $(0, i), (i , n +1)$. Therefore, as a switching solution exists, we have that there exists some $I \subseteq \{1, \dots, n\}$ such that the paths $(0, i), (i , n +1)$ for $i \in I$ are switched on (and others are switched off). Then $\sum_{i \in I} \frac{a_i}{b} = 1$ (by flow conservation at bus $n +1$), the required contradiction.
\end{proof}


%
\section{Valid Inequalities} \label{sec:valid}

In this section, we give two (symmetric) classes of inequalities for
DC-OTS that are derived by considering a relaxation of the cycle
formulation \eqref{Full Cycle}.  The inequalities are derived by
projecting an extended formulation of our chosen relaxation.  We 
additionally show that the inequalities define the convex hull of the
relaxation and that each of the inequalities defines a facet of the
relaxation.  The separation problem for the new class of inequalities
is a knapsack problem, but we show in Section~\ref{sec:separation} how
to exploit the special structure of the knapsack to give a closed-form
solution.

{\bred 
We remind the reader that the objective function is assumed to be linear, as opposed to convex quadratic. 
In general, finding the convex hull of the feasible region may not be as useful algorithmically when the objective
function is nonlinear convex as the optimal solution may lie in the interior of the convex hull. However, the generation
cost functions are usually convex increasing functions of $p_i^g$ over the interval $[p_i^{\text{min}},
p_i^{\text{max}}]$ and, therefore, the optimal solutions of any convex relaxation will lie on the boundary of the relaxation. 
Thus, finding the convex hull of the feasible solutions may still be useful in improving bounds when we are working with a convex quadratic increasing objective function instead of a linear objective function.}
Nevertheless, we also point out that when the optimal solution is not an extreme point, the bound obtained by this approach might be weak. As a future work,  numerical experiments should be carried out to analyze this case empirically.

\subsection{Derivation}

Consider the constraints (\ref{cycleG}), {\bred (\ref{powerLog})
and (\ref{binaryS}) } in the cycle-based formulation for DC-OTS for one
specific cycle $C \in \cC$ and define the following relaxation of the
feasible region of \eqref{Full Cycle}. 
\begin{equation}
\begin{split}
\mathcal{S}_C  &=\{(f, x) :     - M_C   \sum_{ (i,j) \in C}  (1- x_{ij}) \le \sum_{ (i,j) \in C} \frac{f_{ij}}{\bar B_{ij}^C}  \le  M_C   \sum_{ (i,j) \in C}  (1- x_{ij}) ,  \\
    &-\bar f_{ij}x_{ij}  \le f_{ij}   \le \bar f_{ij}x_{ij}   \quad
(i,j) \in  {C}, \ x_{ij} \in \{0,1\}  \quad (i,j) \in  {C} \}.
\end{split}
\end{equation}
In the remainder of this section, we assume that $C$ is a directed cycle and hence, $\bar B_{ij}^C = B_{ij}$.
Our main result in this section concerns the inequalities
\begin{equation}
\begin{split}
-\Delta(S) (|C|-1)+\sum_{ (i,j) \in S}  [\Delta(S) - w_{ij}]x_{ij} +  \Delta(S) \sum_{ (i,j) \in C \setminus S} x_{ij} \\ 
\le \sum_{ (i,j) \in S}  \frac{f_{ij}}{ B_{ij}} \le \quad \quad\quad \quad\quad \quad\quad \quad  \quad \\  
\Delta(S) (|C|-1) -\sum_{ (i,j) \in S}  [\Delta(S) - w_{ij}]x_{ij} -  \Delta(S) \sum_{ (i,j) \in C \setminus S} x_{ij}  \\
  \quad S \subseteq {C}  \ \  \text{s.t.} \ \ \Delta (S) > 0 \label{project without y},
\end {split}
\end{equation}
where $w_{ij} := \frac{\bar f_{ij}}{ B_{ij}}$, and
\begin{alignat*}{2}
w(S) &:= \sum_{(i,j) \in S} w_{ij} \ &&\mbox{ for } S \subseteq C\\
\Delta(S) &:= w(S) - w(C \setminus S) = 2 w(S) - w(C) \ && \mbox{ for } S
\subseteq C .
\end{alignat*}
We show that the inequalities \eqref{project without y} are the
only non-trivial inequalities defining $\text{conv} (\mathcal{S}_C).$
\begin{thm} \label{convex hull}
\begin{equation}
\begin{split}
\text{conv} (\mathcal{S}_C)  = \{(f, x):  (\ref{project without y}),  
    -\bar f_{ij}x_{ij}  \le f_{ij}   \le \bar f_{ij}x_{ij},  \   x_{ij} \le 1  \quad (i,j) \in  {C} \}
\end{split}
\end{equation}
\end{thm}
In proving this result, for ease of presentation, we assume without loss of generality that $B_{ij}=1$ for all $(i,j)\in\mathcal{L}$ by appropriately scaling $\bar f_{ij}$. 
The result is proven through a series of propositions using
disjunctive arguments. Let us start with the following linear system
\begin{subequations} \label{disjunctive}
\begin{align} 
  & \hspace{0.5em}   -\bar f_{ij} x_{ij}^1 \le f_{ij}^1  \le \bar f_{ij} x_{ij}^1  &(&i,j) \in {C}  \\
  & \hspace{0.5em} \sum_{ (i,j) \in C}   x_{ij}^1  = |C| y_C   \\
  & \hspace{0.5em} \sum_{ (i,j) \in C} f_{ij}^1  = 0  \\
  & \hspace{0.5em} 0 \le x_{ij}^1 \le y_C  &(&i,j) \in {C} \\
  & \hspace{0.5em}   -\bar f_{ij} x_{ij}^0 \le f_{ij}^0  \le \bar f_{ij} x_{ij}^0  &(&i,j) \in {C}  \\
  & \hspace{0.5em} \sum_{ (i,j) \in C}   x_{ij}^0  \le (|C| - 1)(1-y_C) \\
  & \hspace{0.5em} 0 \le x_{ij}^0 \le 1-y_C   &(&i,j) \in {C}  \\
  & \hspace{0.5em} x_{ij} = x_{ij}^1  + x_{ij}^0 &(&i,j) \in {C}  \\
  & \hspace{0.5em}  f_{ij} =  f_{ij}^1 +  f_{ij}^0    &(&i,j) \in {C} \\
  & \hspace{0.5em} 0 \le y_C \le 1,
\end{align}
\end{subequations}
and define the polytope 
\[ \mathcal{E}_C = \{(f,  f^1, f^0, x,x^1, x^0, y) :
(\ref{disjunctive})\}. \]

\begin{prop} \label{ cycle prop}
System (\ref{disjunctive}) is an extended
formulation for $\text{conv}(\mathcal{S}_C)$.
Furthermore, polytope $\mathcal{E}_C$ is integral 
so that we have  conv$(\mathcal{S}_C)$ = $\text{proj}_{f, x} \mathcal{E}_C$.
\end{prop}
\begin{proof}
 Let us first consider the following disjunction for  cycle $C$: Either every line is active or at least one line is disconnected. 
If all the lines are active, then we have
\begin{subequations} \label{cycle}
\begin{align} 
  & \hspace{0.5em}   -\bar f_{ij} x_{ij} \le f_{ij}  \le \bar f_{ij} x_{ij}  &(&i,j) \in {C}  \\
  & \hspace{0.5em} \sum_{ (i,j) \in C}   x_{ij}  = |C|  \label{cycle x = 1} \\
  & \hspace{0.5em} \sum_{ (i,j) \in C}  {f_{ij}}   = 0  \\
  & \hspace{0.5em} 0 \le x_{ij} \le 1  &(&i,j) \in {C}
\end{align}
\end{subequations}
Define polytope $\mathcal{S}_C^1 = \{(f, x): (\ref{cycle}) \}$. Note that $\mathcal{S}_C^1$ is integral in $x$ since  constraint (\ref{cycle x = 1}) forces 
$x_{ij}=1$ for all $(i,j) \in C$ in a feasible solution.

Otherwise, at least one of the lines is inactive and we have
\begin{subequations} \label{no cycle}
\begin{align} 
  & \hspace{0.5em}   -\bar f_{ij} x_{ij} \le f_{ij}  \le \bar f_{ij} x_{ij}  &(&i,j) \in {C}  \label{no cycle 1} \\
  & \hspace{0.5em} \sum_{ (i,j) \in C}   x_{ij}  \le |C| - 1 \\
  & \hspace{0.5em} 0 \le x_{ij} \le 1   &(&i,j) \in {C} \label{no cycle 2}
\end{align}
\end{subequations}
Define polytope $\mathcal{S}_C^0 = \{(f, x): (\ref{no cycle}) \}$
{\bred, which is again integral in $x$}. 

By construction, we have $\text{conv}(\mathcal{S}_C) = \text{conv}(\mathcal{S}_C^1 \cup \mathcal{S}_C^0)$. Let us duplicate variables $(f, x)$ as $(f^1, x^1)$ and $(f^0, x^0)$ in the descriptions of $\mathcal{S}_C^1$ and $\mathcal{S}_C^0$, respectively. Then, by assigning a binary variable $y_C$ to $\mathcal{S}_C^1$ and $1-y_C$ to $\mathcal{S}_C^0$, we get system $(\ref{disjunctive})$. 
So, it is an extended formulation for $\text{conv}(\mathcal{S}_C)$.

Further, observe that $\mathcal{E}_C$ is the union of two polyhedra
that are integral in $x$:  $\mathcal{E}_C = \mathcal{S}_C^1 \cup
\mathcal{S}_C^0$.  Therefore, $\mathcal{E}_C$ must be integral in $x$
as well.  
\end{proof}


By noticing that $x_{ij}^1= y_C$ and $x_{ij}^0= x_{ij} - y_C$, we can
simplify the notation by immediately projecting out these variables.
Specifically, if we define the linear system
\begin{subequations} \label{Extended Weak}
\begin{align} 
  & \hspace{0.5em}   -\bar f_{ij} y_C \le f_{ij}^1 \le \bar f_{ij} y_C  &(&i,j) \in {C} \label{powerEW1} \\
  & \hspace{0.5em} \sum_{ (i,j) \in C} f_{ij}^1  = 0  \label{cycleEW} \\ 
  & \hspace{0.5em}   -\bar f_{ij} (x_{ij}-y_C) \le f_{ij}^0  \le \bar f_{ij}(x_{ij}- y_C)  &(&i,j) \in {C} \label{powerEW0} \\
  & \hspace{0.5em} \sum_{ (i,j) \in C}  x_{ij} -y_C \le |C| - 1 \label{cycleEWLogical} \\
  & \hspace{0.5em}   y_C  \le x_{ij} \le1   &(&i,j) \in {C} \label{extended logical} \\
  & \hspace{0.5em}  f_{ij} =  f_{ij}^1 +  f_{ij}^0    &(&i,j) \in {C}  \label{powerArcEW} \\
  & \hspace{0.5em}   0 \le y_C \le 1,   \label{yBound}
\end{align}
\end{subequations}
we have that $\mathcal{P}_C : =
\{(f,f^1,f^0,x,y) : (\ref{Extended Weak})\}= \text{proj}_{f,f^1,f^0,x,y} \mathcal{E}_C$.  In 
Proposition~\ref{prop:pc} we can further project out the $f^1$
and $f^0$ variables by defining
\begin{equation}
-\sum_{ (i,j) \in S}  w_{ij} x_{ij} +\Delta (S) y_C  \le \sum_{ (i,j) \in S} f_{ij} \le \sum_{ (i,j) \in S}  w_{ij} x_{ij} - \Delta (S) y_C  \quad S \subseteq {C}  \ \  \text{s.t.} \ \ \Delta (S) > 0. \label{project with y}
\end{equation}

\begin{prop} 
\label{prop:pc}
$\text{proj}_{f, x, y} \mathcal{P}_{C} =   \{ (f{}, x{}, y) : {\bred (\ref{powerLog}) }, (\ref{project with y}), (\ref{cycleEWLogical}),   (\ref{extended logical}), y_C \ge 0 \} $.
\end{prop}
\begin{proof}  
We begin by defining $\mathcal{Q} :=  \{ (f, x, y) :
{\bred (\ref{powerLog}) } , (\ref{project with y}), (\ref{cycleEWLogical}),
(\ref{extended logical}), y_C \ge 0\} $, and let $( f, f^1, f^0,x,y) \in  \mathcal{P}_{C}$. We claim that $(f, x, y) \in
\mathcal{Q}$. For each line $(i,j) \in {\bred C}$, summing (\ref{powerEW1}) and (\ref{powerEW0}) and using
\eqref{powerArcEW} yields {\bred(\ref{powerLog})}. So, it suffices to check constraint (\ref{project with y}). We have, for each $S \subseteq C$,
$$
0 = \sum_{ (i,j) \in C} f_{ij}^1 = \sum_{ (i,j) \in S} f_{ij}^1 + \sum_{ (i,j) \in C \setminus S} f_{ij}^1
$$
due to (\ref{cycleEW}). Recall that by scaling $\bar{f}_{ij}$, we have assumed $B_{ij}=1$, and thus $w_{ij} =
\bar{f}_{ij}$ for all $(i,j) \in \mathcal{L}$. Combined with (\ref{powerArcEW}), we have 
\begin{align}\begin{split}
\sum_{ (i,j) \in S} f_{ij} &= \sum_{ (i,j) \in S}  {f_{ij}^0}  - \sum_{ (i,j) \in C \setminus S} f_{ij}^1 \\ 
& \le \sum_{ (i,j) \in S} \bar f_{ij} (x_{ij} - y_C) + \sum_{ (i,j) \in C \setminus S} \bar f_{ij} y_C \quad\quad\quad \text{due to (\ref{powerEW0}) and (\ref{powerEW1})} \\
& = \sum_{ (i,j) \in S} w_{ij} x_{ij} - \left( \sum_{ (i,j) \in S}w_{ij} - \sum_{ (i,j) \in C \setminus S}w_{ij} \right ) y_C \\
& = \sum_{ (i,j) \in S} w_{ij} x_{ij} - \Delta(S) y_C.
\end{split}\end{align}
This is exactly the right inequality of (\ref{project with y}).
Note that although this inequality is valid for all $S \subseteq C$, the ones with $\Delta(S) \le 0$ are dominated. In fact, due to  {\bred (\ref{powerLog})} for a subset $\bar S$ with $\Delta(\bar S) \le 0$,  we have
\begin{equation}
\sum_{ (i,j) \in  \bar S}  f_{ij} \le \sum_{ (i,j) \in  \bar S} w_{ij} x_{ij} \le \sum_{ (i,j) \in \bar S} w_{ij} x_{ij}  - \Delta(\bar S) y_C.
\end{equation}
Using a symmetric argument, we can show the validity of the left inequality similarly. 
Hence, $\text{proj}_{f, x, y} \mathcal{P}_{C} \subseteq \mathcal{Q}$.

{\bred Next, we prove that any solution $(f, x, y) \in \mathcal{Q}$ can be extended by some $(f^1, f^0)$ such that it satisfies (\ref{Extended Weak}). First, we can eliminate $f^0$ variables by setting  $f_{ij}^0=f_{ij}-f_{ij}^1$. 
Therefore to show that $\text{proj}_{f,x,y} \mathcal{P}_C\supseteq\mathcal{Q}$, it suffices to show that, for any $(f,x,y)\in\mathcal{Q}$, the following system in $f_{ij}^1$ is always feasible,
\begin{subequations} \label{eq:f1}
	\begin{align}
	& \hspace{0.5em}   -w_{ij} y_C \le f_{ij}^1 \le w_{ij} y_C  &(&i,j) \in {C}  \label{eq:f1a}\\
	& \hspace{0.5em}   f_{ij}-w_{ij} (x_{ij}-y_C) \le  f_{ij}^1  \le f_{ij}+w_{ij}(x_{ij}- y_C)  &(&i,j) \in {C} \label{eq:f1b}\\
	& \hspace{0.5em} \sum_{ (i,j) \in C} f_{ij}^1  = 0, \label{eq:f1c}
	\end{align}
\end{subequations} 
which is the intersection of a hypercube defined by \eqref{eq:f1a}-\eqref{eq:f1b} and a hyperplane \eqref{eq:f1c}. First note that the hypercube is always nonempty due to \eqref{powerLog}. Then, due to the continuity of the function $\sum_{(i,j)\in C}{f_{ij}^1}$, it suffices to show the following inequalities:
\begin{align} \label{eq:minmaxf1}
\min\biggl\{\sum_{(i,j)\in C} f_{ij}^1 : \eqref{eq:f1a}, \eqref{eq:f1b}\biggr\}\leq 0 \leq \max\biggl\{\sum_{(i,j)\in C} f_{ij}^1 : \eqref{eq:f1a}, \eqref{eq:f1b}\biggr\}.
\end{align}
We show that \eqref{eq:minmaxf1} follows from \eqref{project with y}. To see this, let us look at the right inequality in \eqref{eq:minmaxf1}. The left one can be obtained symmetrically. Since \eqref{eq:f1a}-\eqref{eq:f1b} define a hypercube, the maximum in \eqref{eq:minmaxf1} is obtained by maximizing each $f_{ij}^1$ over its own interval, i.e. we have
\begin{align}
\max\biggl\{\sum_{(i,j)\in C} f_{ij}^1 : \eqref{eq:f1a}, \eqref{eq:f1b}\biggr\} &= \sum_{(i,j)\in C}\min\biggl\{w_{ij}y_C, f_{ij}+w_{ij}(x_{ij}-y_C)\biggr\}\notag \\
&= \sum_{(i,j)\in S} f_{ij}+w_{ij}(x_{ij}-y_C) + \sum_{(i,j)\in C\setminus S} w_{ij}y_C \notag \\
&=\sum_{(i,j)\in S} f_{ij} + w_{ij}x_{ij} - \Delta(S)y_C \geq 0, \label{eq:fS}
\end{align}
where $S=\{(i,j)\in C : f_{ij}+w(x_{ij}-y_C)\le w_{ij}y_C\}$ and \eqref{eq:fS} is exactly the left-hand side inequality in \eqref{project with y}.}
\end{proof}

We complete the proof of Theorem~\ref{convex hull} by projecting out
the $y_C$ variable as well.

\begin{proof} [Proof of Theorem \ref{convex hull}]
By construction, we have that $\text{conv} (\mathcal{S}_C)  =  \text{proj}_{f, x} \mathcal{P}_{C}$.
It suffices to show that  $\text{conv}(\mathcal{S}_C) = \text{proj}_{f, x} \mathcal{P}_{C} =  \mathcal{R}$ where $ \mathcal{R}= \{ (f, x) : {\bred (\ref{powerLog})}, (\ref{project without y}),  x_{ij}  \le 1  \ \ (i,j) \in C \} $.



First, let us rewrite inequalities involving $y_C$ in $ \text{proj}_{f, x, y}  \mathcal{P}_{C}$:
\begin{subequations}
\begin{align}
0 \le &y_C  \label{eq2_0} \\
\sum_{(i,j) \in C} x_{ij} - (|C|-1) \le &y_C \label{eq2_1} \\
&y_C \le x_{ij} & (&i,j) \in C  \label{eq2_2} \\
&y_C \le  \frac{1}{\Delta (S) } \sum_{ (i,j) \in S}  w_{ij} x_{ij}  + \frac{1}{\Delta (S) }  \sum_{ (i,j) \in S} f_{ij}  & S& \subseteq {C}, \Delta (S) > 0 \label{eq2_3} \\
&y_C \le \frac{1}{\Delta (S) } \sum_{ (i,j) \in S}  w_{ij} x_{ij}-\frac{1}{\Delta (S) }  \sum_{ (i,j) \in S} f_{ij}    & S& \subseteq {C}, \Delta (S) > 0. \label{eq2_4}
\end{align}
\end{subequations}
Now, we use Fourier-Motzkin elimination on $y_C$ to show necessity and sufficiency of the convex hull description. Note that equation (\ref{eq2_2}) together with (\ref{eq2_0}) and (\ref{eq2_1}) give redundant inequalities dominated by $0 \le x_{ij} \le 1$. Similarly,  (\ref{eq2_0}) with  (\ref{eq2_3}) and (\ref{eq2_4}) exactly give equation {\bred (\ref{powerLog})}.

Next, we look at (\ref{eq2_1}) and (\ref{eq2_4}) together, which give
\begin{align}\begin{split}
& \quad\quad\quad \sum_{(i,j) \in C} x_{ij} - (|C|-1)  \le \frac{1}{\Delta (S) } \sum_{ (i,j) \in S}  w_{ij} x_{ij}  - \frac{1}{\Delta (S) }  \sum_{ (i,j) \in S} f_{ij}  \\
& \iff \Delta (S)\sum_{(i,j) \in C} x_{ij} - \Delta (S)(|C|-1)  \le  \sum_{ (i,j) \in S}  w_{ij} x_{ij}  -   \sum_{ (i,j) \in S} f_{ij} \\
& \iff   \sum_{ (i,j) \in S} f_{ij}  \le \Delta (S)(|C|-1) +   \sum_{ (i,j) \in S}  w_{ij} x_{ij} - \Delta (S)\sum_{(i,j) \in S} x_{ij} - \Delta (S)\sum_{(i,j) \in C \setminus S} x_{ij}    \\
& \iff   \sum_{ (i,j) \in S} f_{ij}  \le \Delta (S)(|C|-1)  -   \sum_{ (i,j) \in S} [ \Delta(S) - w_{ij}] x_{ij}  - \Delta (S)\sum_{(i,j) \in C \setminus S} x_{ij}, 
\end{split}\end{align}
for $ S \subseteq {C}, \Delta (S) > 0$. But, this is the  right inequality in  (\ref{project without y}).
Finally, using a similar argument, if we look at (\ref{eq2_1}) and (\ref{eq2_3}) together, we get 
 exactly  left inequality in  (\ref{project without y}) for $ S \subseteq {C}, \Delta (S) > 0$, which concludes the proof.
\end{proof}


Theorem~\ref{convex hull} shows that the inequalities \eqref{project
  without y} are sufficient to define the convex hull of $\cS_C$.  In
Theorem~\ref{facet_prop} we show that each inequality is also
necessary by showing that they always define facets.  
\begin{thm} \label{facet_prop} 
The inequalities \eqref{project without y} are facet-defining for the set
$\text{conv} (\mathcal{S}_{C})$.
\end{thm}

\begin{proof}
We prove the result only for the right inequality of \eqref{project
  without y}, as a symmetric argument can be made for the left inequality. 
First, it is easy to establish that $\text{conv} (\mathcal{S}_{C})$ is
full-dimensional by  producing $2|C|+1$ affinely independent points that lie
in $(\mathcal{S}_{C})$.  
Let $\mathcal{F}$ denote the face of $\text{conv} (\mathcal{S}_{C})$ defined by
\begin{align*}
\mathcal{F} = \Bigl\{(f,x) \in \mathcal{S}_{C} : \sum_{(i,j)\in S} f_{ij} +
\sum_{(i,j) \in S} [\Delta (S) - w_{ij}] x_{ij} + \Delta (S)
\sum_{(i,j) \in C \setminus S} x_{ij} = \Delta (S) (|C|-1) \Bigr\}.
\end{align*} 
We produce $2|C|$ affinely independent points in $\cF$ for the case $|C \setminus S| \neq 1$ 
and $2|C|-1$ linearly independent points in $\cF$ for the case $|C \setminus S| =1$, which
establishes the result.

\exclude{
Jeff removed this.  The notation was all wrong, and we can put it back
if we need to do the algebra to show that the points given later are
in fact in $\cF$

The following conditions are satisfied by points $ (\bar{f}, \bar{x}) \in \mathcal{F}$.
\begin{itemize}
\item If $\bar{x}_{ij} = 1 \ \forall (i,j) \in C$: 
\begin{align*}
 \bar{f}(S) &+ |S| \Delta (S) - w(S) + (|C| - |S|) \Delta (S) = \Delta (S) (|C|-1) \\
 \Leftrightarrow \bar{f}(S) &= w(S) - \Delta (S) = w(S) - (w(S) - w(S^C)) \\
                            &= w(S^C)
\end{align*} 
Since $\bar{f}(C) = 0$ in this case, it also implies that $f_{ij} = -w_{ij} \ \forall (i,j) \in S^C$. 
\item If $\bar{x}_{ij} = 1 \ \forall (i,j) \in C \setminus a, \ \bar{x}_a =0$ for some $a \in S$: 
 \begin{align*}
 \bar{f}(S) &+ (|S|-1) \Delta (S) - w(S) + w_a + (|C| - |S|) \Delta (S) = \Delta (S) (|C|-1) \\
 \Leftrightarrow \bar{f}(S) &= w(S) - w_a 
\end{align*} 
\item If $\bar{x}_{ij} = 1 \ \forall (i,j) \in C \setminus a, \ \bar{x}_a =0$ for some $a \in S^C$: 
 \begin{align*}
 \bar{f}(S) &+ (|S|-1) \Delta (S) - w(S) + (|C| - |S|) \Delta (S) = \Delta (S) (|C|-1) \\
 \Leftrightarrow \bar{f}(S) &= w(S) 
\end{align*}
\end{itemize}
}

Let $\hat{a} \in S$, define $\rho = w(C \setminus S)/w(S) \in [0,1)$,
  let $\epsilon$ be a sufficiently small real number, and define the
  $|S|$ points $\chi^1(\hat{a})$ whose components take values
\begin{align*}
f_a &= \left\{ \begin{array}{ll}  
\rho w_a + \epsilon  & \text{if } a \in S \setminus \hat{a}  \\
\rho w_a - (|S|-1) \epsilon  & \text{if } a = \hat{a} \\
- w_a                        & \text{if } a \in C \setminus S \\
\end{array}  \right. \\
x_a &= 1 \ \forall a \in C.
\end{align*}
Define an additional $|S|$ points $\chi^2(\hat{a})$ whose components
take the values
\begin{align*}
f_a &= \left\{ \begin{array}{l l}  
w_a        & \text{if } a \in C \setminus \hat{a} \\
0          & \text{if } a = \hat{a} \\
\end{array}  \right. \\
x_a &= 1 \ \forall a \in C \setminus \hat{a},\\
x_{\hat{a}} &= 0. 
\end{align*}
One can establish that both $\chi^1(\hat{a}), \chi^2(\hat{a}) \in \cF
\ \forall \hat{a} \in S$. 
If $S = C, ( \text{i.e., }|C \setminus S| = 0)$ these constitute $2|C|$ affinely independent points. 
Otherwise, the proof continues by defining 2 $|C \setminus S|$ points in $\cF$. 

For each $\tilde{a} \in C \setminus S$, define a point $\chi^3(\tilde{a})$ 
whose components take values 
\begin{align*} 
  f_a &= \left\{ \begin{array}{l l}  
                               w_a        & \text{if } a \in S \\
                               0          & \text{if } a \in C \setminus S \\
                               \end{array}  \right. \\
  x_a &= 1 \ \forall a \in C \setminus \tilde{a}, \\
  x_{\tilde{a}} &= 0. 
\end{align*}
If $|C \setminus S| \geq 2$, we can define an additional point 
$\chi^4(\tilde{a},\bar{a})$ for some $\bar{a} \in C \setminus S, \ \bar{a} \neq \tilde{a}$ whose components take values  
\begin{align*}
  f_a &= \left\{ \begin{array}{l l}  
                               w_a        & \text{if } a \in S \cup \bar{a} \\
                               0          & \text{if } a \in C \setminus S \setminus \bar{a} \\
                               \end{array}  \right. \\
  x_a &= 1 \ \forall a \in C \setminus \tilde{a}, \\
  x_{\tilde{a}} &= 0.
\end{align*}
It can be shown that both $\chi^3(\tilde{a}), \chi^4(\tilde{a}, \bar{a}) \in \cF \ \forall \tilde{a}, \bar{a} \in C \setminus S$. 
Then the set of points $\{\chi^1(\hat{a}), \chi^2(\hat{a}): \hat{a} \in S\} \cup 
\{\chi^3(\tilde{a}), \chi^4(\tilde{a}, \bar{a}) : \text{ each } a \in C \setminus S \text{ is chosen as } \tilde{a}, \bar{a} \text{ exactly once} \}$ 
is in $\cF$ and contains $2|C|$ points that are affinely independent.  

In the case $|C \setminus S| = 1$, we construct $2|C|-1$ linearly independent points in $\cF$. 
Consider the set of points 
$\{\chi^1(\hat{a}), \chi^2(\hat{a}): \hat{a} \in S\} \cup \{\chi^3(\tilde{a}) : \tilde{a} \in C \setminus S \} \subseteq \cF$ 
and let $p^1, p^2, \cdots, p^{2|C|-1} \in \mathbb{R}^{2|C|}$ denote its $2|C|-1$ elements. 
We prove their linear independence by showing that $\sum_{i = 1}^{2|C|-1} \lambda_i p^i = {\bf 0}, \ \lambda \in \mathbb{R}^{2|C|-1} $  
implies $\lambda_i = 0 \ \forall i = 1, 2, \cdots, 2|C|-1$. 
For simplicity of notation, we use numerical index for edges contained in cycle. 
Let $C = \{1,2, \cdots, n\}$ and $S = \{1,2, \cdots, n-1\}$. 
Then $\sum_{i = 1}^{2|C|-1} \lambda_i p^i = {\bf 0}$ can be equivalently written as the following system of equations: 
\begin{align}
 \sum_{i=1}^{n-1} (\rho w_k + \epsilon) \lambda_i - (n-1) \epsilon \lambda_k + 
  \sum_{i=n}^{2n-1} \lambda_i w_k - \lambda_{n+k} w_k = 0 \quad \forall k = 1,2, \cdots, n-1, \label{eq:lincom_f} \\
 \sum_{i=1}^{2n-1} \lambda_i - \lambda_k = 0 \quad \forall k = n, n+1, \cdots, 2n-1. \label{eq:lincom_x}
 \end{align}
From the $n$ inequalities in \eqref{eq:lincom_x}, 
we obtain $\lambda_{n} = \lambda_{n+1} = \cdots = \lambda_{2n-1} = \sum_{i=1}^{2n-1} \lambda_i = \bar{\lambda}$ 
for some number $\bar{\lambda}$ which implies $\sum_{i=1}^{n-1} \lambda_i = (1-n)\bar{\lambda} $.  
Plugging the former equations into \eqref{eq:lincom_f} results in  
\begin{align*}
 \lambda_k = (((1-\rho)w_k - \epsilon)/\epsilon)\bar{\lambda} \quad \forall k = 1,2, \cdots, n-1. 
\end{align*}
Aggregating these $n-1$ inequalities, we obtain 
\begin{align*}
 & \sum_{k=1}^{n-1} \lambda_i = \sum_{k=1}^{n-1}((1-\rho)/\epsilon)w_k \bar{\lambda} - \sum_{k=1}^{n-1} \bar{\lambda} = (1-n)\bar{\lambda} \\
 & \Leftrightarrow ((1-w(C \setminus S) / w(S)) w(S) / \epsilon = 0 \\
 & \Leftrightarrow ((2w(S) - w(C))/\epsilon)\bar{\lambda} = 0 \\
 & \Leftrightarrow \bar{\lambda} = 0 \\
 & \Rightarrow \lambda_i = 0 \ \forall i = 1, 2, \cdots, 2|C|-1. 
\end{align*}
\end{proof}

\subsection{Separation}
\label{sec:separation}

\newcommand{\fa}{f_{ij}}
\newcommand{\xa}{x_{ij}}
\newcommand{\wa}{w_{ij}}
\newcommand{\DS}{\Delta(S)}
\newcommand{\fh}{\hat{f}}
\newcommand{\xh}{\hat{x}}
\newcommand{\Kc}{K_C}
\newcommand{\viol}{\mathrm{viol}}
\newcommand{\wt}{\nu}
\newcommand{\wtm}{\bar{\nu}}

We now investigate the {separation} problem over constraints (\ref{project without y}).
Given a fractional point $(\fh, \xh)$, let us first define $K_C =  1 - \sum_{ (i,j) \in C} (1- \xh_{ij}) $. 
We focus on the right inequality in (\ref{project without y}); the left is analyzed similarly. For the given cycle $C$
and $S \subseteq C$ with $\Delta(S) > 0$ the inequality takes the form:
\begin{equation*}
 \sum_{ (i,j) \in S}  \fa + \sum_{ (i,j) \in S}  [\DS - \wa]\xa + {\bred \Delta(S) } \sum_{ (i,j) \in {\bred C }} \xa \leq  {\bred \Delta(S) }  (|C|-1) .
\end{equation*}
For convenience, we rearrange the inequality as follows:
\begin{equation}
\label{validrear}
 \sum_{ (i,j) \in S} (\fa - \wa \xa) + \Delta(S) \Bigl ( {1 - \sum_{ (i,j) \in C} (1- x_{ij})} \Bigr) \leq 0 .
\end{equation}
Our aim is to determine if there exists $S \subseteq C$ with $\Delta (S) > 0$ such that
\begin{equation}
\viol(S):= \sum_{ (i,j) \in S} (\fh_{ij} - w_{ij} \xh_{ij}) + \Delta (S) K_C  > 0 .
\end{equation}
Recalling that $\Delta(S) = w(S) - w(C \setminus S) = 2 w(S) - w(C)$, this can be determined by solving 
\begin{equation}
\label{sep}
\max_{S \subseteq C} \Bigl\{ \sum_{ (i,j) \in S} (\fh_{ij} -  w_{ij} \xh_{ij}) + \Delta (S) K_C : \Delta (S) > 0 \Bigr\}. 
\end{equation}
A violated inequality exists if and only if the optimal value of \eqref{sep} is positive. Introducing binary variables
$z_{ij}$ for $(i,j) \in C$ to indicate whether or not line $(i,j) \in S$, \eqref{sep} can be reformulated as follows:
\begin{equation} \label{knapsack separate}
\begin{split}
& \max_{z \in \{0,1\}^{|C|}} \Bigl\{ \sum_{ (i,j) \in C} (\fh_{ij} - w_{ij} \xh_{ij}) z_{ij} + 
\sum_{ (i,j) \in C}  w_{ij}K_C z_{ij} - \sum_{ (i,j) \in C}  w_{ij}K_C(1-z_{ij}): 
\\ & \quad\quad\quad\quad\quad\quad\quad\quad\quad\quad\quad\quad\quad\quad \sum_{ (i,j) \in C}  w_{ij}z_{ij} - \sum_{ (i,j) \in
C}  w_{ij} (1-z_{ij}) > 0 \  \Bigr\}  \\
 & = -w(C)K_C + \max_{z \in \{0,1\}^{|C|}} \Bigl\{ \sum_{ (i,j) \in C} \left ( \fh_{ij} -  w_{ij} \xh_{ij} + 2w_{ij}K_C  \right ) z_{ij} : 
 \sum_{ (i,j) \in C}  w_{ij}z_{ij}  >  \frac12 w(C) \Bigr\} .
\end {split}
\end{equation}
Note that here we do need the condition that $ \Delta (S) > 0$, which yields a knapsack problem. A similar minimization problem can be posed to separate left inequalities.

There is a necessary condition for a cycle $C$ to have a violating inequality of form (\ref{project without y}) given in Proposition \ref{necessary K}.
\begin{prop} \label{necessary K}
Given $(\fh, \xh)$ for a cycle $C$, if $K_C =  1- \sum_{ (i,j) \in C} (1- \xh_{ij}) \le 0$, then  inequalities  (\ref{project without y}) are not violated.
\end{prop}
\begin{proof}
Consider right inequalities first. 
If $K_C < 0$, then given an optimal solution $z$ to problem (\ref{knapsack separate}), we have
 $\sum_{ (i,j) \in C} \left ( \fh_{ij} -  w_{ij} \xh_{ij} + 2w_{ij}K_C  \right ) z_{ij} < w(C) K_C$ since $\sum_{ (i,j) \in
 C}  w_{ij}z_{ij}  >  \frac12 w(C)$ and ${f_{ij}} \le  w_{ij} x_{ij}$. Therefore, no violating inequality exists.
If $K_C = 0$, then given an optimal solution $z$ to problem (\ref{knapsack separate}), we have $\sum_{ (i,j) \in C} \left
( \fh_{ij} -  w_{ij} \xh_{ij} + 2w_{ij}K_C  \right ) z_{ij} \le  0 = w(C) K_C$ since ${\fh_{ij}} \le  w_{ij} \xh_{ij}$. Therefore, no violating inequality exists.
A similar argument can be used to show that the requirement $K_C > 0$ is necessary for a left inequality to be violating.
\end{proof}

Although \eqref{knapsack separate} formulates the separation problem of inequalities \eqref{project without y} as a
knapsack problem, the special structure of this knapsack problem enables it to be solved efficiently.  In fact, we could
solve a linear program derived from the extended formulation \eqref{Extended Weak} to solve the separation problem over  
$\text{conv} (\mathcal{S}_C)$.  However, we next show how separation of the cycle inequalities can be accomplished
efficiently in closed form.  Define 
\[ S^*_C = \{ (i,j) \in C : \fh_{ij} - \wa \xh_{ij} + 2\wa \Kc > 0 \}. \]
The following proposition shows that $S^*_C$ is the only subset that needs to be considered when solving the separation
problem for cycle $C$.

\begin{prop}
\label{sopt}
Assume $\Kc > 0$. If there is any $S \subseteq C$ with $\Delta(S) > 0$ and $\viol(S) > 0$, then the separation problem \eqref{sep} is
solved by $S^*_C$.
\end{prop}
\begin{proof}
First, suppose that $\Delta(S^*_C) > 0$, so that $S^*_C$ is a feasible solution to \eqref{sep}.  Then, since by
construction it contains every element in $C$ that has a positive contribution to the objective $\viol(S)$, this is an
optimal solution to \eqref{sep}.  Now, suppose $\Delta(S^*_C) \leq 0$, i.e., $2w(S^*_C) - w(C) \leq 0$.  Then,
\[ \viol(S^*_C) = \sum_{(i,j) \in S^*_C} (\fh_{ij} - \wa \xh_{ij}) + (2w(S^*_C) - w(C))\Kc \leq 0 \]
because the first term is nonpositive due to $\fh_{ij} \leq \wa \xh_{ij}$ and the second term is nonpositive by assumption.  Now,
because each element $(i,j) \in S^*_C$ contributes a positive amount to the objective in \eqref{sep} and also a positive
amount to the constraint, there must exist an optimal solution of \eqref{sep} that contains $S^*_C$ as a subset. But,
any set $T$ that contains  $S^*_C$ as a subset has $\viol(T) \leq \viol(S^*_C) \leq 0$, since 
elements $(i,j) \notin S^*_C$ contribute a non-positive amount to the objective.  Thus, in the case $\Delta(S^*_C) < 0$,
there is no $S \subseteq C$ with $\Delta(S) > 0$ and $\viol(S) > 0$.
\end{proof}

Recall from Proposition \ref{necessary K}  that $\Kc > 0$ is a necessary condition for a violated inequality to exist
from cycle $C$.  Thus, for a given cycle $C$ with $\Kc > 0$, a violated inequality exists for this cycle if and only if:
\begin{equation}
\label{violcond}
 \sum_{(i,j) \in C}\Bigl( (\fh_{ij} - \wa \xh_{ij} + 2 \wa \Kc)_+ - w_{ij} \Kc \Bigr) > 0.
 \end{equation}
 where $(\cdot)_+ = \max \{\cdot,0\}$.
The separation problem then reduces to a search for a cycle $C$ having $\Kc > 0$ and that satisfies \eqref{violcond}. 

\section{Algorithms} \label{section:alg}

In this section, we describe our algorithmic framework for solving the DC switching problem, beginning with section \ref{sec:overview}
where we describe the overall algorithm. 
Section \ref{sec:separationalg} describes algorithms we implemented for separating over cycle inequalities \eqref{project without y}
for a fixed cycle $C$, and section \ref{sec:generatecycles} describes our procedure for generating a set of cycles over
which we will perform the separation. 


\subsection{Overall Algorithm}
\label{sec:overview}
The overall structure of the proposed algorithm is shown in Algorithm~\ref{alg:OverallAlgorithm}. The preprocessing
phase of the algorithm aims to add cycle inequalities to strengthen the LP relaxation of the Angle Formulation
\eqref{eq:Fisher-Form}. In particular, we first generate a set of cycles over the original power network, and solve the
LP relaxation of \eqref{eq:Fisher-Form}. Then, a separation algorithm finds all the cycle inequalities \eqref{project
without y} that are violated by the LP solution over the generated set of cycles. These violated inequalities are added
to the LP relaxation as cuts. This procedure is iterated for a number of times to strengthen the LP relaxation, which is
then fed to the MIP solver (we use CPLEX). We prefer to implement our separation in this ``cut-and-branch'' manner in
order to investigate the utility of the cutting planes when combined with all advanced features of modern MIP software.
An alternative approach would be to the use User Cut callback facility of CPLEX. However, this procedure disables many
advanced features of the solver such as dynamic search.  {\bred We implemented a version of our cutting planes using the CPLEX callback features and found that on average the performance was around 30\% worse than with our cut-and-branch Algorithm~\ref{alg:OverallAlgorithm}.}
\begin{algorithm}
\caption{Overall Algorithm for DC Switching}
\label{alg:OverallAlgorithm}
\begin{enumerate}
\item Preprocessing: 
	\begin{enumerate}
		\item Generate a set $\Gamma$ of cycles (Cycle basis generation Algorithm \ref{Cycle basis generation.}).
          \item  Strengthen LP relaxation of Angle Formulation \eqref{eq:Fisher-Form} by adding violated cycle inequalities \eqref{project without y} from each cycle $C \in \Gamma$ (Separation Algorithms \ref{alg:separation} or \ref{more cuts heuristic.}).
	\end{enumerate}
\item Solve the Angle Formulation with added cuts using CPLEX.
\end{enumerate}
\end{algorithm}

\subsection{Separation Algorithms}\label{sec:separationalg}

For a given LP relaxation solution, the separation algorithm implements the ideas presented in Section~\ref{sec:separation} to identify all violated cycle constraints of the form \eqref{project without y} for a predetermined set of cycles. The procedure is summarized in Algorithm \ref{alg:separation}. 

\begin{algorithm}
\caption{Separation Algorithm}
\label{alg:separation}
\begin{algorithmic}
\STATE Given a set $\Gamma$ of cycles and a LP relaxation solution $(\fh, \xh)$. 
\FOR {every cycle $C\in\Gamma$}
\STATE Compute $\Kc =  1 - \sum_{a \in C} (1 - \xh_{a})$.
\IF {$\Kc > 0$}
\STATE For each $a\in C$, compute $z_{a}=\begin{cases} 1 &  \text{if }  \fh_a -  w_{a} \xh_{a} + 2w_{a}\Kc  \ge 0 \\ 0 & \text{otherwise} \end{cases}$
	\IF {$\sum_{a\in C} w_a z_a>\frac{1}{2}w(C)$ \AND $\sum_{a \in C}\Bigl( (\fh_a - w_{a} \xh_a + 2 w_{a} \Kc)_+ - w_a \Kc \Bigr) > 0$}
		\STATE A violated cycle inequality for $C$ is found.
	\ENDIF
\ENDIF 
\ENDFOR
\end{algorithmic}
\end{algorithm}


Algorithm \ref{alg:separation} generates a {\it single} violating inequality for each cycle, if such a violated inequality exists.  However, the method can be extended to find all violating inequalities for a cycle. 
This procedure is summarized in  Algorithm \ref{more cuts heuristic.}, which uses a recursive subroutine described in Algorithm \ref{recursion.}.

\begin{algorithm}
\caption{Finding all valid inequalities.}
\label{more cuts heuristic.}
\begin{algorithmic}
\STATE Given a cycle $C$, define $v_{ij} = \fh_{a}-  w_{a} \xh_{a} + 2w_{a}\Kc $    for $a \in C$
\STATE Set $S = \{a \in C : v_{a} \ge 0\} $ and denote $C \setminus S = \{a_1, \dots, a_n\}$
\STATE Calculate $v(S) = \sum_{a \in S} v_{a}$ and $w(S) = \sum_{a \in S} w_{a}$
\STATE Recursion($S$, $0$)
\end{algorithmic}
\end{algorithm}

\begin{algorithm}
\caption{Recursion($S$, $k$)}
\label{recursion.}
\begin{algorithmic}
\IF{ $v(S) \le w(C)\Kc $ or $k = n$}
\STATE  Stop.
\ENDIF
\IF{ $v(S) >  w(C)\Kc $ and  $w(S) >\frac12w(C) $}
\STATE  A violating inequality is found.
\ENDIF
\FOR{$l=k+1, \dots, n$}
\STATE Recursion($S \cup \{a_l\}$, $l$)
\ENDFOR
\end{algorithmic}
\end{algorithm}

\subsection{Cycle Basis Algorithm}\label{sec:generatecycles}
{\bred
This section discusses how to generate a set of cycles for the preprocessing phase of Algorithm~\ref{alg:OverallAlgorithm}.  The number of cycles in
a graph $G=(V,E)$ grows exponentially in $|V|$, so in computations, finding all cycles is not efficient. Instead, we find  a cycle basis for the original power network and use the cycle basis to generate cycles for separation.  There are many algorithms for finding a cycle basis \cite{kavitha.et.al:09}.
We use a simple algorithm based on the LU decomposition of the incidence matrix of the graph $G$.  This choice of algorithm was arbitrary.  We have no reason to believe that this specific choice of algorithm to create a cycle basis has a significant impact on performance of cutting planes.   For completeness, we state the algorithm below.
}


Consider a directed graph $G=(V, E)$ with vertex set $V$ and arc set $A$. Let $|V|=n$ and $|E|=m$.
We define edge-node incidence matrix $A$ as 
\begin{equation}  \label{DefineA matrix}
A_{(i,j),k} = \begin{cases}
1 & \text{if } i=k \\
-1 & \text{if } j=k \\
0 &	\text{otherwise}
\end{cases}
\end{equation}
Then, assuming that $G$ is connected, Algorithm \ref{Cycle basis generation.} can be used to find a cycle basis. The correctness of the algorithm is proved in Appendix \ref{app:generatecycles}.
\begin{algorithm} 
	\caption{Cycle basis generation.}
	\label{Cycle basis generation.}
	\begin{algorithmic}
		\STATE Define edge-node incidence matrix $A$ of directed graph $G$ as given in (\ref{DefineA matrix}).
		\STATE Carry out LU decomposition of A with partial pivoting to compute $PA = LU$ with  a unit lower triangular matrix $L$.
		\STATE Last $m - n + 1$ rows of $L^{-1}P$, denoted by $C_b$, gives a cycle basis.
	\end{algorithmic}
\end{algorithm}

{\bred 
Given an initial set of cycles $\Gamma$ coming from the cycle basis, the following
procedure is used to generate additional cycles from which we may apply the separation procedures \ref{alg:separation} and 
\ref{more cuts heuristic.}.  Any pair of cycles in $\Gamma := \mathcal{C}^0$ that share at least one common edge can be combined to form a new cycle by removing the common edges. Denote $\bar{\mathcal{C}^0}$ as the set of all the new cycles thus generated from $\mathcal{C}^0$. Then, the set $\mathcal{C}^1:=\mathcal{C}^0\cup\bar{\mathcal{C}^0}$ has more cycles than the cycle basis $\mathcal{C}_b$. This process can be repeated to generate sets $\mathcal{C}^{k+1}:= \mathcal{C}^k\cup\bar{\mathcal{C}^k}$ for $k\geq 1$.}

Given a set of cycles $\Gamma$, we can use Algorithms \ref{alg:separation} or \ref{more cuts heuristic.} to identify and add all violated cycle
inequalities for that set of cycles  to the LP relaxation of the Angle Formulation. We solve this strengthened LP
relaxation again and add further violated cycle inequalities. This procedure can be carried out in several iterations (five times for our experiments) to
produce a strengthened LP relaxation that is eventually passed to the MIP solver. 

\exclude{

\begin{algorithm}
\caption{Strengthening of LP relaxation.}
\label{LP relaxation.}
\begin{algorithmic}
\STATE Given a set $\Gamma$ of cycles.
\STATE $r\gets$ 1  \ \quad //iteration count
\STATE $\mathcal{V} \gets \emptyset$  \quad //list of valid inequalities
 \WHILE{$r \le R$}
\STATE Solve LP relaxation of Angle Formulation \eqref{eq:Fisher-Form}.
\STATE Find violated inequalities \eqref{project without y} using Algorithm \ref{alg:separation} or \ref{more cuts heuristic.}.
\STATE Store identified inequalities in the list $\mathcal{V}_r$.
\IF{ $\mathcal{V}_r = \emptyset$}
\STATE  $r\gets R$.
\ELSE
\STATE Add  cuts  in $\mathcal{V}_r$ to LP relaxation.
\STATE Update $\mathcal{V} \gets \mathcal{V} \cup \mathcal{V}_r$.
\ENDIF
\STATE Set $r\gets r+1$.
 \ENDWHILE
\STATE Add cuts in $\mathcal{V}$ to MIP formulation \eqref{eq:Fisher-Form}.
\end{algorithmic}
\end{algorithm}
}

\section{Computational Results} \label{section:comp}
In this section, we present extensive computational studies that demonstrate the effectiveness of our proposed cut-and-branch algorithms on the DC-OTS problem.  Section \ref{Instance Generation} explains how the test instances are generated. 
Section \ref{DCswitching comp} compares the default branch-and-cut algorithm of CPLEX with two algorithms that employ
the cycle inequalities \eqref{project without y}.   The first algorithm generates inequalities from cycles in one fixed
cycle basis, and the other generates inequalities from a larger set of cycles.  The results show that the proposed
algorithm with cutting planes separated from more cycles consistently outperforms the default algorithm in terms of the
size of the branch-and-bound tree, the total computation time, and the number of instances solved within the time limit. 

For all experiments, we use a single thread in a 64-bit computer with Intel Core i5 CPU 3.33GHz processor and 4 GB RAM. Codes are written in the  \texttt{C\#} language. 
Considering that the transmission switching problem is usually solved under a limited time budget, the relative
optimality gap is set to a moderate amount of $0.1\%$ for all MIPs solved using CPLEX 12.4 (\cite{Cplex}).  We set a time
limit of one hour in all experiments.


\subsection{Instance Generation}\label{Instance Generation}

Our computational experiments focus on instances where the solution of
the DC-OTS is significantly different than the solution of the DC-OPF.
In addition to selecting instances where switching made an appreciable
instance, we selected instances whose network size was large enough so
that the instances were not trivial for existing algorithms, but small
enough to not be intractable.  The 118-bus instance \textsf{case118B}
generated in \cite{Blumsack} turns out to be suitable for our
purposes, and we also modified the 300-bus instance \textsf{case300}
so that transmission switching produces meaningfully different
solutions from the OPF problem without switching.  Furthermore, in
order to extensively test the effectiveness of the proposed cuts and
separation algorithms, we generate the following five sets of
instances based on \textsf{case118B} and \textsf{case300}:

\begin{itemize}
\item
\text{Set 118\_15}: We generate 35 instances by modifying \textsf{case118B}, where the load at each bus of the original \textsf{case118B} is increased by a discrete random variable following a uniform distribution on $[0,15]$.
\item
\text{Set 118\_15\_6}: To each of the instances in \text{Set 118\_15}, we randomly add 5 new lines to the power network, each creating a 6-cycle. The transmission limits for the lines in the cycle are set to 30\%  of the smallest capacity $\bar f_{ij}$ in the network, and $B_{ij}$ is chosen randomly from one of the lines which is already in the original network.  
\item
\text{Set 118\_15\_16}: The instances are constructed the same as in
\text{Set 118\_15\_6}, except that a 16-cycle is created by adding 5
new transmission lines.
\item
\text{Set 118\_9G}:   
We generate 35 different instances from \textsf{case118B}, where the
original load at each bus is increased by a discrete random variable
following a uniform distribution on $[0,9]$. Furthermore, the
generation topology of the network is changed. In particular, a
generator located at bus $i$ is moved to one of its neighboring buses
or stays at its current location with equal probability. 
\item
\text{Set 300\_5}: We generate 35 different instances from \textsf{case300}, where the original load at each node is
incremented by a discrete random variable following a uniform distribution on $[-5, 5]$. Also, eight generators are turned
off and {\bred the cost coefficients of remaining generators are updated to be similar to the objective coefficients in  \cite{Blumsack}}.
 Finally, more restrictive transmission line limits are imposed.

\end{itemize}


These instances can be downloaded from {\url{https://sites.google.com/site/burakkocuk/research}}.

\subsection{DC Transmission Switching}\label{DCswitching comp}

We now investigate the computational impact of using the proposed valid inequalities within the proposed cut-and-branch
procedure. We compare the following three solution procedures:
\begin{enumerate}
\item
 The angle formulation \eqref{eq:Fisher-Form} solved with CPLEX, abbreviated by Default.
\item
The angle formulation \eqref{eq:Fisher-Form} with valid inequalities
\eqref{project without y} found using cycles coming from a single cycle basis, abbreviated by BasicCycles.
\item
{\bred   
  The angle formulation \eqref{eq:Fisher-Form} with valid inequalities \eqref{project without y} coming from more cycles than a cycle basis, abbreviated by MoreCycles.  The procedure for generating additional cycles for separation is discussed in Section~\ref{sec:generatecycles}.  We use the set of cycles $\mathcal{C}^2$ for the 118-bus networks, where $|\mathcal{C}^2| \approx 3500$. For the 300-bus networks, $\mathcal{C}^2$ has more than $37,000$ cycles, which makes the separation procedure quite computationally expensive.  For the 300-bus networks, we select $10\%$ of the cycles in $\mathcal{C}^2$ randomly for separation.}
\end{enumerate}
{\bred We conducted preliminary experiments comparing Algorithms \ref{alg:separation} and \ref{more cuts heuristic.} 
when using the valid inequalities \eqref{project without y}, and found Algorithm \ref{more cuts heuristic.} yielded
consistently better performance. Therefore, we use Algorithm \ref{more cuts heuristic.} as the separation algorithm for inequalities
\eqref{project without y} in both BasicCycles and
MoreCycles.}

Tables \ref{tb:118-15}-\ref{tb:300-5} show the computational results
for the five sets of test instances described in Section \ref{Instance
  Generation}.  {\bred To measure the impact of the cuts on closing the integrality gap, we use the following objective values: 
  \begin{itemize}
    \item $z_{LP}$: the objective value of the
      LP relaxation at the root node without inequalities \eqref{project without y} and without CPLEX cuts;
    \item $z_{LP}^{cuts}$: the objective value of the LP relaxation at the root node with inequalities \eqref{project without y} and without CPLEX cuts; 
    \item $z_{LP}^{root}$: the objective value of the LP relaxation at the root node with inequalities \eqref{project without y} and with CPLEX cuts; 
    \item $z_{IP}$: the objective value of the final integer solution of
      the switching problem.
  \end{itemize}
The integrality gap measures reported in the tables are defined as
\begin{itemize}  
    \item ``Initial LP Gap (\%)'' := $100\%\times (z_{IP}-z_{LP})/z_{IP}$;
    \item ``Gap Closed by Cuts (\%)'' := $100\%\times (z_{LP}^{cuts}-z_{LP})/(z_{IP}-z_{LP})$;
    \item ``Root Gap Closed (\%)'' := $100\%\times (z_{LP}^{root}-z_{LP})/(z_{IP}-z_{LP})$.
\end{itemize}
Other performance metrics reported in the tables are the average number of valid inequalities generated 
by the proposed algorithm (row ``\# Cuts''); the average preprocessing time
for generating valid inequalities (``Preprocessing Time''), which
includes the time for solving five rounds of the LP relaxation of the
switching problem and the associated separation problems; 
the average total computation time including the
preprocessing time (``Total Time''); the number of Branch-and-Bound
nodes (``B\&B Nodes''); the number of unsolved instances within a time
limit of one hour (``\# Unsolved''); and the average final optimality gap for unsolved instances (``Unsolved Opt Gap
(\%)''). For each metric, we report both the arithmetic mean (the first number) and the geometric mean (the second number).

}

From these tables, we can see that the proposed algorithm MoreCycles consistently outperforms the default algorithm in terms of the percentage of optimality gap closed at the root note, the size of the Branch-and-Bound tree, the total computation time, and the number of instances solved. 

Figures \ref{fig:118-15}-\ref{fig:300-5} show the performance profiles of the three algorithms on the five sets of test instances. In particular, each curve in a performance profile is the cumulative distribution function for the ratio of one algorithm's runtime to the best runtime among the three (\cite{Dolan}). 
Set 118\_15 is a relatively easy test set. Figure \ref{fig:118-15}
shows that for $42.9\%$ of the instances, the Default algorithm is the
fastest algorithm, the BasicCycles algorithm is fastest on $37.1\%$,
and the MoreCycles algorithm is fastest on $20\%$. However, if we
choose being within a factor of two of the fastest algorithm as the
comparison criterion, both BasicCycles and MoreCycles surpass
Default.  BasicCycles solves all the instances and has the dominating performance for this set of instances. 

Figure \ref{fig:118-9G} shows the results for Set 118\_9G. BasicCycles is the fastest
algorithm in $40\%$ of the instances; MoreCycles and Default have the success rate of $20\%$ of being the fastest.  If
we choose being within a factor of four of the fastest algorithm as the interest of comparison, then MoreCycles starts to outperform BasicCycles. Also, MoreCycles solves $74.3\%$ of the instances, which is the highest among the three. 

For instance sets 118\_15\_6 and 118\_15\_16, Figures \ref{fig:118-15-6}-\ref{fig:118-15-16} show that BasicCycles is the fastest algorithm in the highest percentage of instances. For the ratio factor higher than one, BasicCycles and MoreCycles clearly dominate Default, and both solve  significantly more instances than Default. MoreCycles is the most robust algorithm in the sense that it  solves the most instances. 

On the 300\_5 instance set, Figure \ref{fig:300-5} shows that MoreCycles is the fastest in the largest fraction of
instances and it also solves the most instances.   BasicCycles is dominated by the other two methods for this set of instances. 

Figure \ref{fig:all} shows the performance profiles of the three algorithms over all the five test sets. It shows that,
BasicCycles is the fastest algorithm in $38.3\%$ of the instances, whereas Default is the fastest in $29.7\%$  of the instances and MoreCycles is the fastest in $25.7\%$ of the instances. If we look at the algorithm that can solve $75\%$ of all the instances with the highest efficiency, then BasicCycles and MoreCycles have almost identical performance, and both significantly dominate Default. MoreCycles solves slightly more instances than BasicCycles within the time limit. In summary, the performance profiles show that BasicCycles has the highest probability of being the fastest algorithm and MoreCycles solves the most instances.  These experiments demonstrate that the cycle inequalities
\eqref{project without y} can be quite useful in improving the performance of state-of-the-art MIP software for solving the DC-OTS.

\newpage
\vspace{-1cm}
\begin{table}[H] 
\centering
	\begin{subtable}[h]{4.5in}
		\begin{center}
		\small
		\begin{tabular}{rccc}
			\hline
			&    Default &      BasicCycles &    MoreCycles \\
			\hline
			\# Cuts &          - & 32.91/31.64 & 218.66/190.10 \\
			Preprocessing Time (s) &          - &  0.05/0.04 &  0.80/0.44 \\
			Gap Closed by Cuts (\%) &          - &     1.50/0 &     2.90/0 \\
			Root Gap Closed (\%) &     4.41/0 &  7.84/7.32 & 18.43/17.33 \\
			Total Time (s) & 437.75/35.39 & 26.12/15.71 & 234.43/30.51 \\
			B\&B Nodes & 6.1E+5/3.6E+4 & 3.6E+4/1.5E+4 & 2.7E+5/2.2E+4 \\
			\# Unsolved &          2 &          0 &          1 \\
			Unsolved Opt Gap (\%) &  0.30/0.29 &   0/0 &  0.11/0.11 \\
			\hline
		\end{tabular}  
		\end{center}
		\caption{Summary of results for Set 118\_15.  {\bred Initial LP Gap (\%): 26.87/26.57}}
		\label{tb:118-15}
	\end{subtable}
	
	\vspace{2mm}
	\begin{subtable}[h]{4.5in}
		\small
		\begin{tabular}{rccc}
			\hline
			&    Default &     BasicCycles &       MoreCycles \\
			\hline
			\# Cuts &          - & 28.37/27.63 & 150.31/139.02 \\
			Preprocessing Time (s) &          - &  0.05/0.04 &  1.14/0.32 \\
			Gap Closed by Cuts (\%) &          - &     5.43/0 &    10.83/0 \\
			Root Gap Closed (\%) &    13.44/0 &    22.26/0 &    27.24/0 \\
			Total Time (s) & 1126.54/148.65 & 1170.22/129.52 & 951.47/121.90 \\
			B\&B Nodes & 1.9E+6/2.1E+5 & 1.9E+6/1.8E+5 & 1.3E+6/1.5E+5 \\
			\# Unsolved &         10 &         10 &          7 \\
			Unsolved Opt Gap (\%) &  0.71/0.45 &  0.79/0.36 &  0.50/0.33 \\
			\hline	
		\end{tabular} 
		\caption{Summary of results for Set 118\_9G.  {\bred Initial LP Gap (\%): 19.12/13.37}}
		\label{tb:118-9G}
	\end{subtable}
	
	\vspace{2mm}
	
	\begin{subtable}[h]{4.5in} 
		\begin{center}
			\small
			\begin{tabular}{rccc}
				\hline
				&    Default &      BasicCycles &       MoreCycles \\
				\hline
				\# Cuts &          - & 29.34/28.91 & 145.20/141.54 \\
				Preprocessing Time (s) &          - &  0.11/0.03 &  1.30/0.44 \\
				Gap Closed by Cuts (\%) &          - &     1.50/0 &     2.92/0 \\
				Root Gap Closed (\%) &     5.42/0 &  7.80/7.32 & 18.19/17.00 \\
				Total Time (s) & 901.31/124.72 & 506.40/55.70 & 515.05/72.37 \\
				B\&B Nodes & 1.3E+6/1.4E+5 & 4.5E+5/4.9E+4 & 6.1E+5/6.6E+4 \\
				\# Unsolved &          5 &          3 &          1 \\
				Unsolved Opt Gap (\%) &  1.31/0.76 &  1.87/1.68 &  0.13/0.13 \\
				\hline
			\end{tabular}  
		\end{center}
		\caption{Summary of results for Set 118\_15\_6.  {\bred Initial LP Gap (\%): 26.76/26.48}}
		\label{tb:118-15-6}
	\end{subtable}
\end{table}

\begin{table}
\centering
	\begin{subtable}[h]{4.5in} 
		\begin{center}
			\small
			\begin{tabular}{rccc}
				\hline
				&    Default &     BasicCycles &       MoreCycles \\
				\hline
				\# Cuts &          - & 26.54/25.85 & 86.23/84.23 \\
				Preprocessing Time (s) &          - &  0.11/0.05 &  0.54/0.31 \\
				Gap Closed by Cuts (\%) &          - &     0.05/0 &     0.31/0 \\
				Root Gap Closed (\%) &     0.47/0 &     4.38/0 &    11.65/0 \\
				Total Time (s) & 2243.71/1750.82 & 1473.52/924.64 & 1581.60/1170.37 \\
				B\&B Nodes & 2.0E+6/1.5E+6 & 1.2E+6/8.1E+5 & 1.2E+6/9.0E+5 \\
				\# Unsolved &         13 &          4 &          3 \\
				Unsolved Opt Gap (\%) &  0.54/0.41 &  0.93/0.74 &  1.22/0.73 \\
				\hline
			\end{tabular}  
		\end{center}
		\caption{Summary of results for Set 118\_15\_16.  {\bred Initial LP Gap (\%): 3.25/2.81}}
		\label{tb:118-15-16}
	\end{subtable}
	
	\vspace{2mm}
	
	\begin{subtable}[h]{4.5in}
		\begin{center}
			\small
			\begin{tabular}{rccc}
				\hline
				&    Default &      BasicCycles &       MoreCycles \\
				\hline
				\# Cuts &          - & 15.66/15.26 & 34.83/33.56 \\
				Preprocessing Time (s) &          - &  0.09/0.06 &  0.48/0.43 \\
				Gap Closed by Cuts (\%) &          - &  7.26/7.25 &  7.26/7.25 \\
				Root Gap Closed (\%) &  7.11/4.17 & 48.37/48.28 & 48.39/48.30 \\
				Total Time (s) & 1685.39/634.75 & 1940.16/841.88 & 1524.14/514.76 \\
				B\&B Nodes & 6.6E+5/2.3E+5 & 7.8E+5/3.1E+5 & 6.2E+5/1.9E+5 \\
				\# Unsolved &         13 &         16 &        10 \\
				Unsolved Opt Gap (\%) &  0.21/0.19 &  0.22/0.21 &  0.40/0.23 \\
				\hline
			\end{tabular}  
		\end{center}
		\caption{Summary of results for Set 300\_5.  {\bred Initial LP Gap (\%): 4.88/4.88}}
		\label{tb:300-5}
	\end{subtable}
	
	\vspace{2mm}
	
	\begin{subtable}[h]{4.5in} 
		\begin{center}
			\small
			\begin{tabular}{rccc}
				\hline
				&    Default &      BasicCycles &       MoreCycles \\
				\hline
				\# Cuts &          - &26.57/25.10 &127.05/101.12 \\
				Preprocessing Time (s) &          - &  0.08/0.05&  0.85/0.38 \\
				Gap Closed by Cuts (\%) &          - &  3.15/0 & 4.84/0 \\
				Root Gap Closed (\%) & 6.17/0 &18.13/0 & 24.78/0 \\
				Total Time (s) & 1278.94/235.82 & 1023.29/154.57 & 961.34/174.58 \\
				B\&B Nodes & 1.3E+6/2.1E+5 & 8.8E+5/1.3E+5 & 7.9E+5/1.3E+5 \\
				\# Unsolved &         43 &         33 &         22 \\
				Unsolved Opt Gap (\%) &  0.56/0.35 &  0.63/0.34 &  0.52/0.28 \\
				\hline
			\end{tabular}  
		\end{center}
		\caption{Summary of all the instances. {\bred Initial LP Gap (\%): 16.17/10.53}}
		\label{tb:all}
	\end{subtable}
	\caption{\normalsize Summary of computational results. For each metric, we report the arithmetic  and  geometric mean, respectively.}  
\end{table}

\begin{figure}
	\begin{center}
	\begin{subfigure}[h]{3in}
		\centering
		\includegraphics[width=3in]{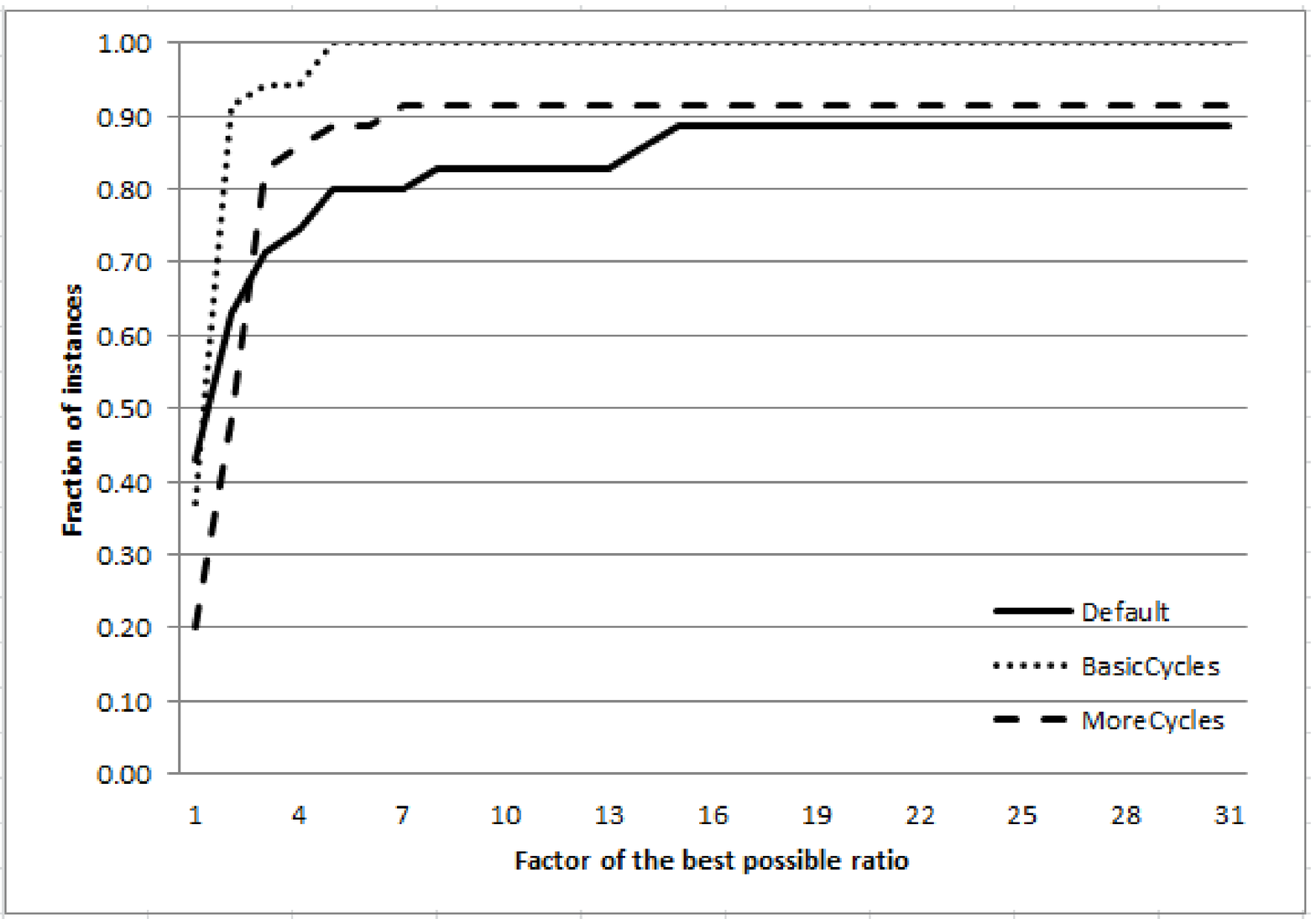}
		\caption{Performance profile for Set 118\_15.} 
		\label{fig:118-15}
	\end{subfigure}
	\;
	\begin{subfigure}[h]{3in}
		\centering
		\includegraphics[width=3in]{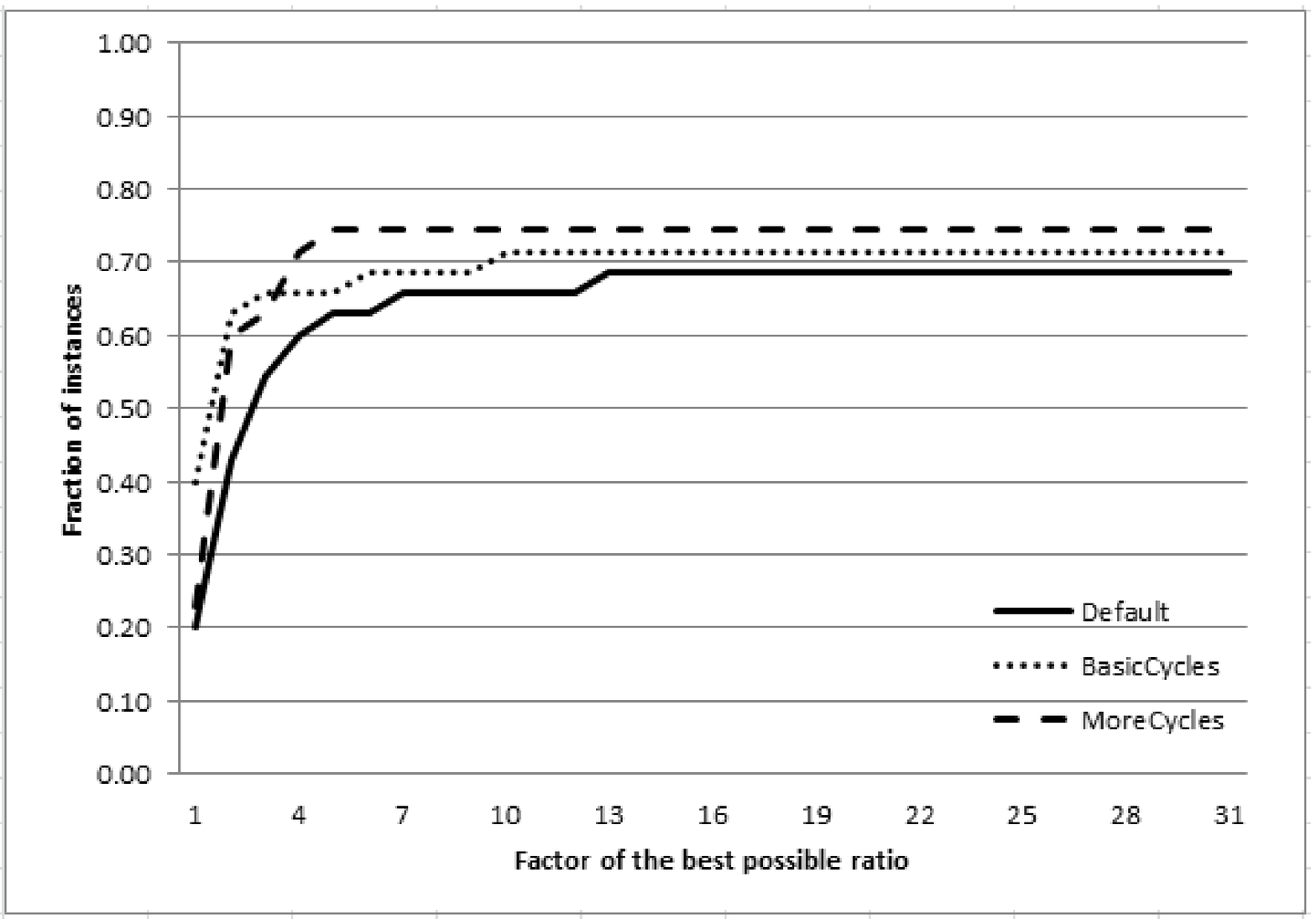}
		\caption{Performance profile for Set 118\_9G.} 
		\label{fig:118-9G}
	\end{subfigure}
	\vspace{3mm}
	
	\begin{subfigure}[b]{3in}
		\includegraphics[width=3in]{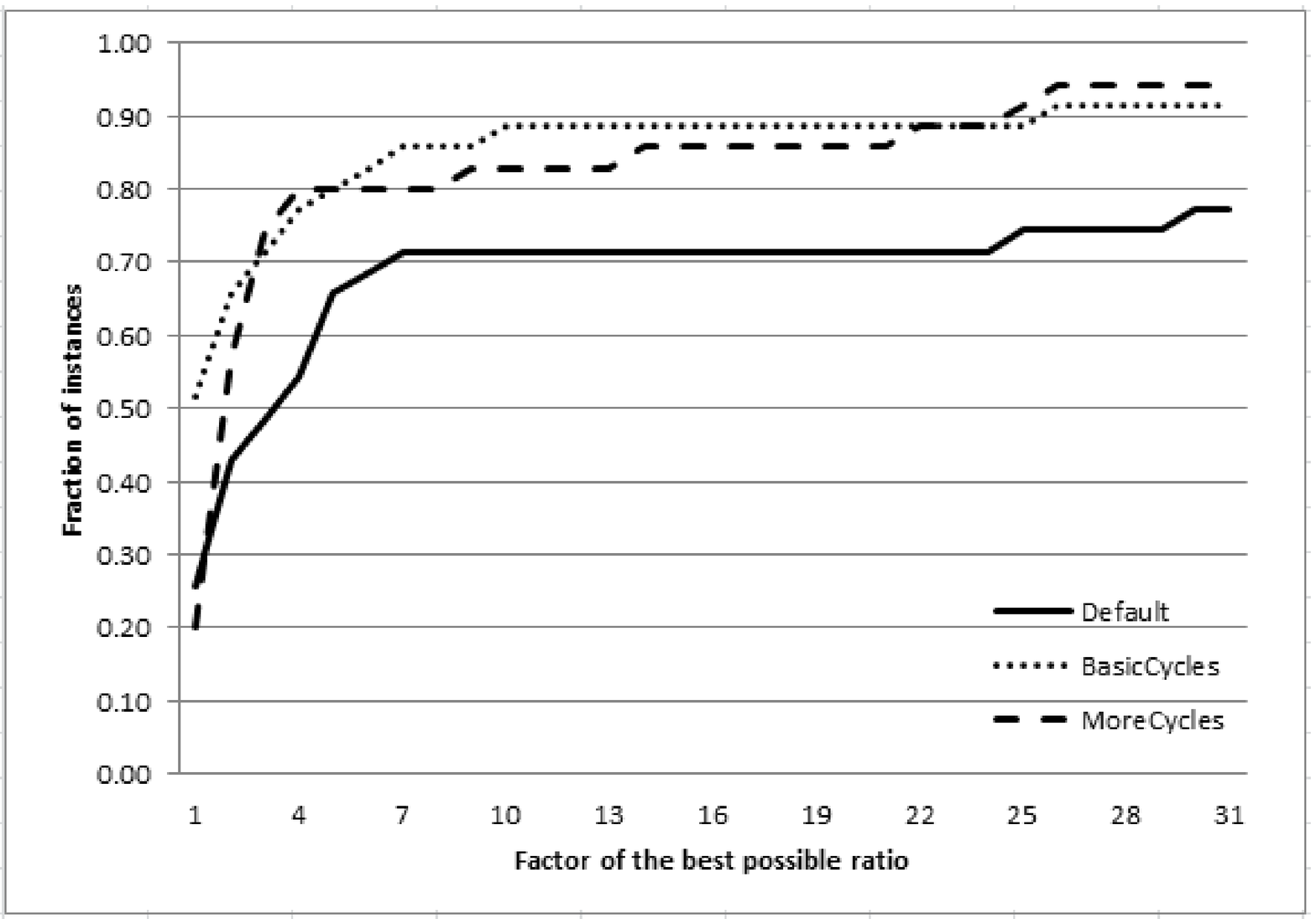}
		\caption{Performance profile for Set 118\_15\_6.} 
		\label{fig:118-15-6}
	\end{subfigure}
	\;
	\begin{subfigure}[b]{3in}
		\includegraphics[width=3in]{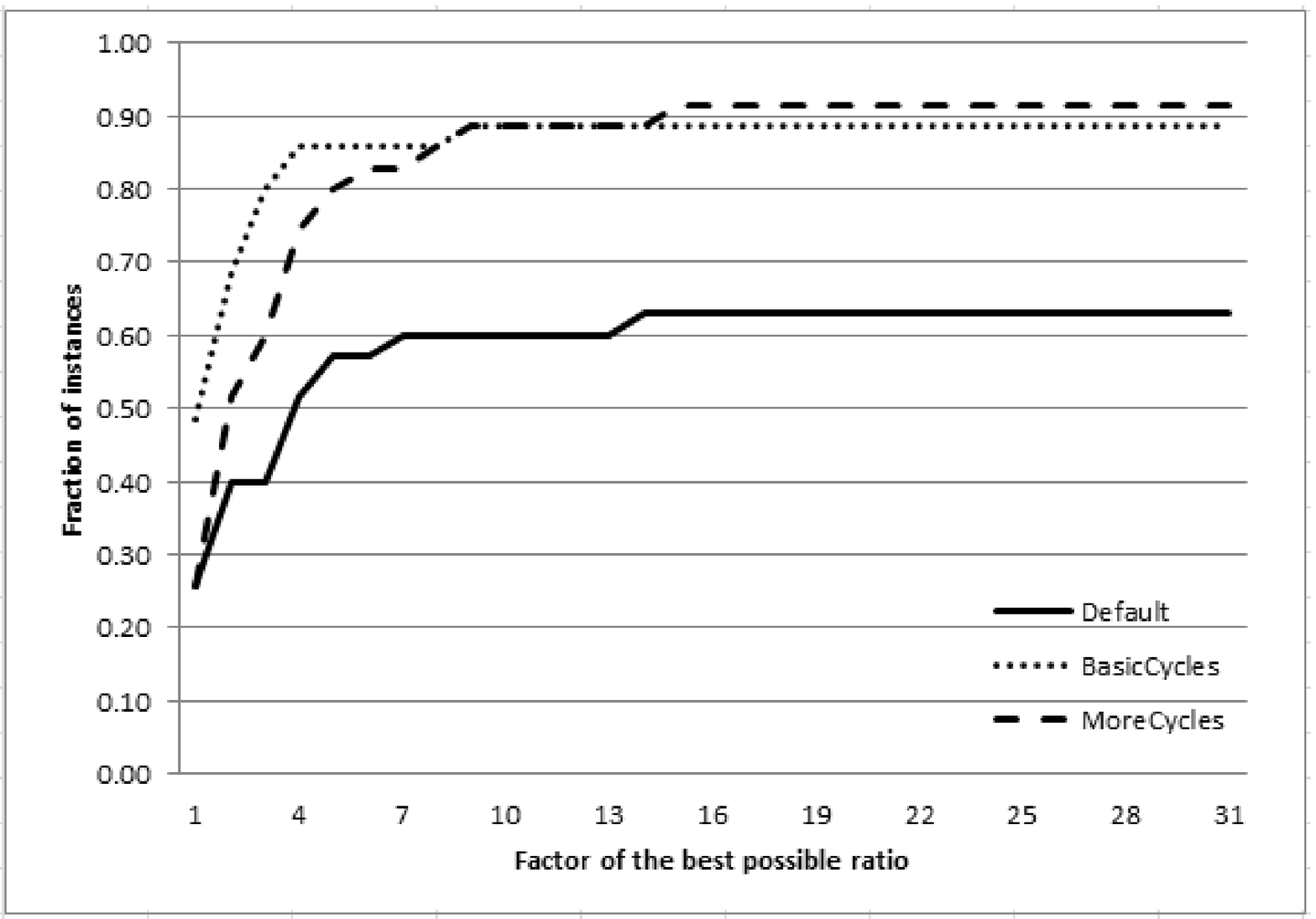}
		\caption{Performance profile for Set 118\_15\_16.} 
		\label{fig:118-15-16}
	\end{subfigure}
	\vspace{3mm}
	
	\begin{subfigure}[b]{3in}
		\includegraphics[width=3in]{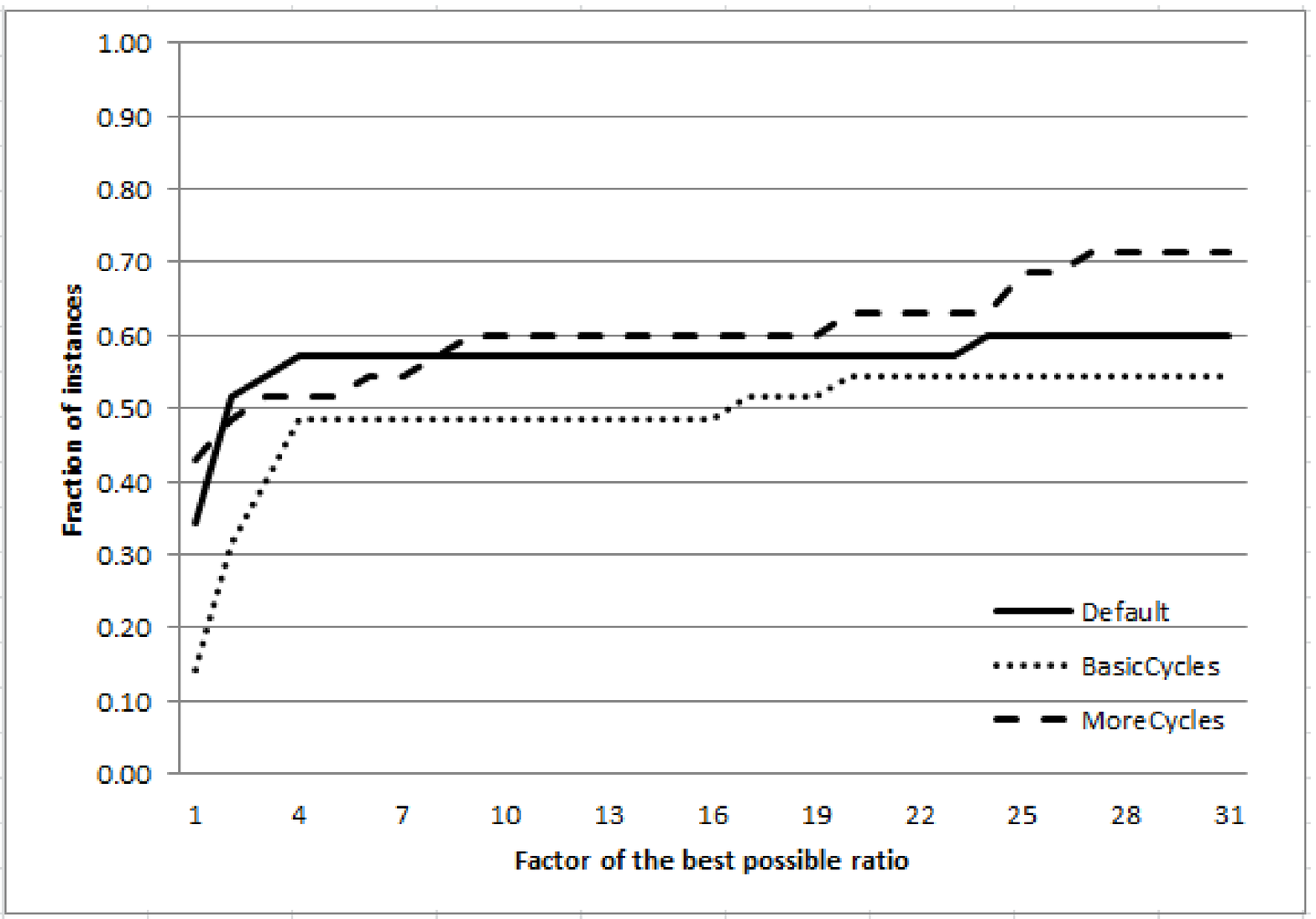}
		\caption{Performance profile for Set 300\_5.} 
		\label{fig:300-5}
	\end{subfigure}
	\;
	\begin{subfigure}[b]{3in}
		\includegraphics[width=3in]{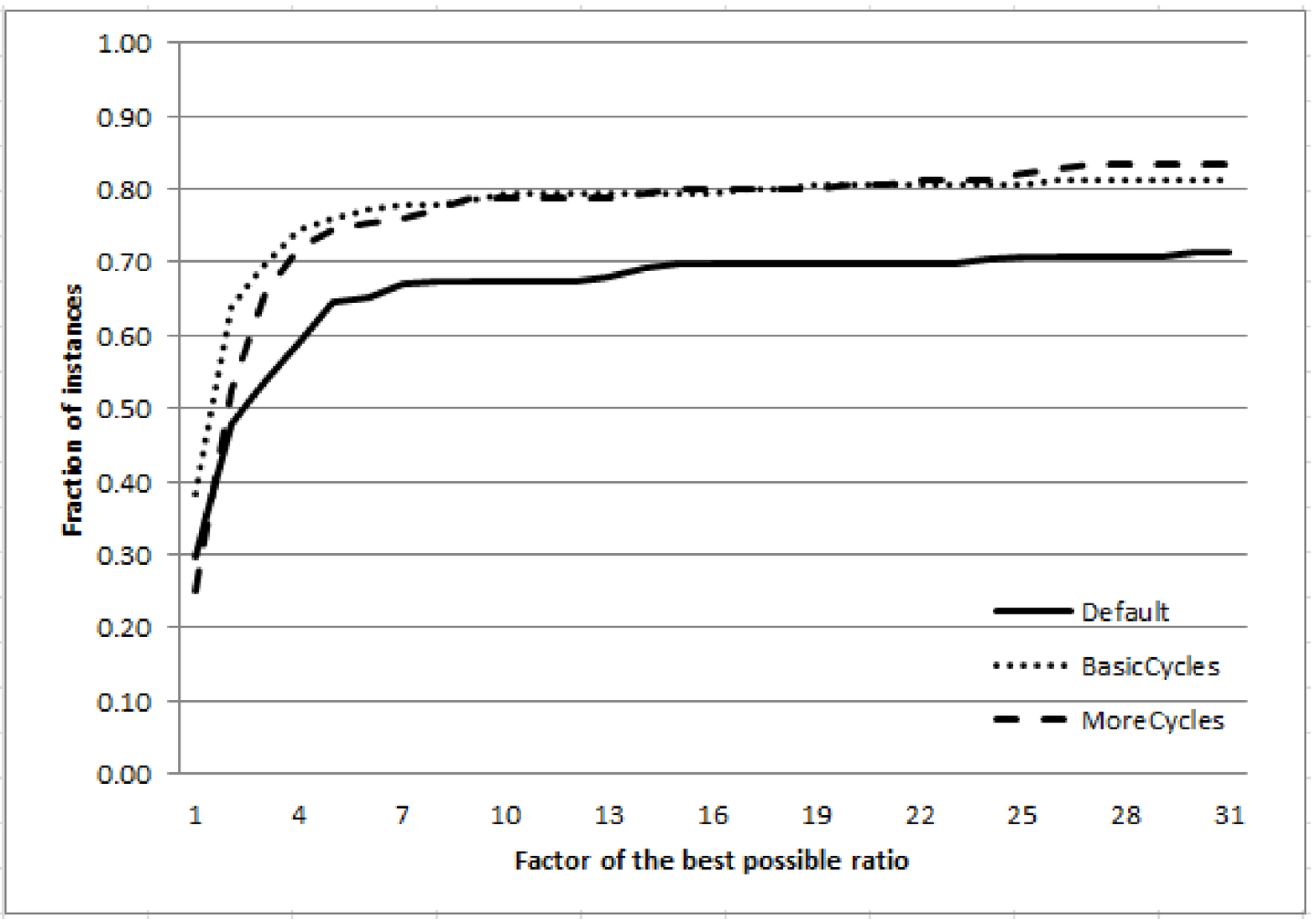}
		\caption{Performance profile for all the instances.} 
		\label{fig:all}
	\end{subfigure}
	\end{center}
	\caption{\normalsize Performance profiles for test instances.}
\end{figure}

\newpage
{\bred
  \subsection{Sensitivity Analysis}
  Previous literature on the DC-OTS demonstrates that significant cost
  savings can be achieved by switching off only a few lines
  \citep{Fisher,Wu13}.  However, the optimal solutions obtained by the
  integer program turn off a significantly larger number of lines than
  suggested by previous studies.  Specifically, the average number of
  lines turned off in the optimal solutions to the 118-bus instances
  is 42, and the maximum number turned off is 57.  For the 300-bus
  instances, an average of 85 lines are turned off in the optimal
  solutions, with a maximum of 107.  This surprising result is a
  consequence of our observation that there are {\it many} optimal or
  near-optimal topologies for the DC-OTS.  To demonstrate the impact
  of switching off fewer lines than suggested by the optimal solution
  to the integer program, we performed a sensitivity analysis of the
  cost versus the number of lines $N$ that are allowed to be switched
  off.
  \exclude{
  
  Our computations indicate that there are many topologies that 
An important practical issue in OTS is to determine the number of lines to be switched off in order to obtain most of the cost saving. Previous literature suggests this can be achieved by switching off only a few lines. 
It turns out that there are several different network topologies, which give practically the same objective function value and without any cardinality constraints, Angle Formulation tends to pick a topology with several lines switched off. 
Since these numbers are relatively large compared to the size of the
networks, one may be interested in a cost-effective topology with
fewer lines switched off. }

In this analysis, we chose one instance from each of the five sets whose
optimal solution had a large number of lines switched off (41, 38, 41,
48, 91, respectively - refered to as instances (a), (b), (c), (d), (e)
henceforth) and solved a number of DC-OTS instances with the
cardinality constraint: 
\begin{equation}
\sum_{(i,j) \in \mathcal{L}} x_{ij} \ge |\mathcal{L}| - N
\end{equation}
added to the formulation.
Here, $N$ is the \textit{switching budget}, that is, the number of lines allowed to be switched off (note that $N=0$ corresponds to DC-OPF). We experimented with different $N$ values and the results are given in Figures \ref{fig:card118-15}-\ref{fig:card300-5}. We make the following observations:
\begin{itemize}
\item
DC-OPF versions of instances (a), (c), (d) and (e) are infeasible.
\item
Once a particular instance becomes feasible, increasing the switching budget has a significant effect on the objective value for the first few lines (especially, for instances (b), (d) and (e)).
\item
Nevertheless, the full cost benefit can only be realized by switching off several lines.
\item
Switching off 11 lines is enough for 118-bus instances (a), (b), (c), (d) to achieve the maximum cost benefit (just
seven lines are needed for instance (b)). 
\item For the 300-bus instance, switching off 15 lines yields nearly the maximum cost benefit. 
\item The LP relaxation value is not affected by the switching budget.
\end{itemize}
Our results support the observation that most, although not all, of the
cost benefits in transmission switching can be realized by switching
off only a few lines.  This has a positive impact on the robustness of
the network.  In our experience, the MIP instances with cardinality
constraints were more time consuming to solve that without the
cardinality constraint.  For example the two instances which had the largest number of switched off lines in the optimal
solutions we found (the 118-bus instance with 57 lines switched off and the
300-bus instance with 107 lines switched off) could not be solved in one hour for most switching budgets. Our 
observation that the LP relaxation value is not affected by the switching budget may help explain this. Since the
optimal IP value is larger for smaller switching budgets, the LP relaxation gap is larger for these instances.
\begin{figure}
	\begin{center}
	\begin{subfigure}[h]{3in}
		\centering
		\includegraphics[width=3in]{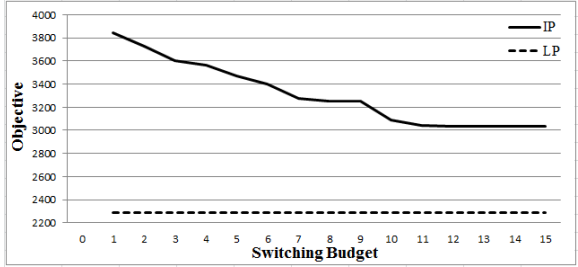}
		\caption{An instance from Set 118\_15.} 
		\label{fig:card118-15}
	\end{subfigure}
	\;
	\begin{subfigure}[h]{3in}
		\centering
		\includegraphics[width=3in]{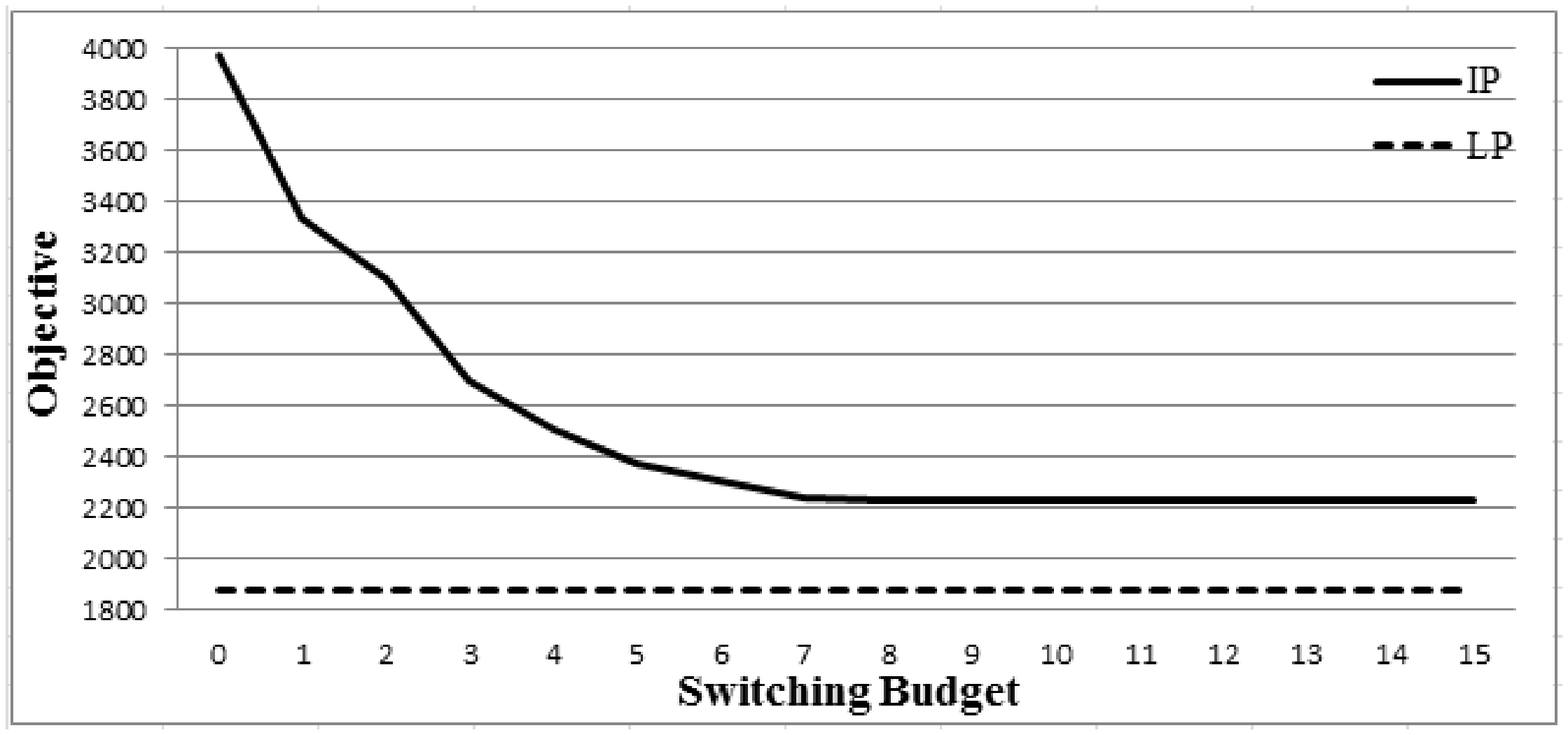}
		\caption{An instance from Set 118\_9G.} 
		\label{fig:card118-9G}
	\end{subfigure}
	\vspace{3mm}
	
	\begin{subfigure}[b]{3in}
		\includegraphics[width=3in]{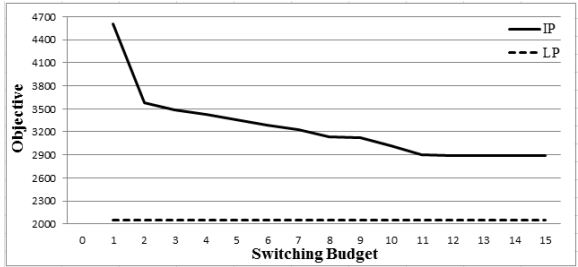}
		\caption{An instance from Set 118\_15\_6.}  
		\label{fig:card118-15-6}
	\end{subfigure}
	\;
	\begin{subfigure}[b]{3in}
		\includegraphics[width=3in]{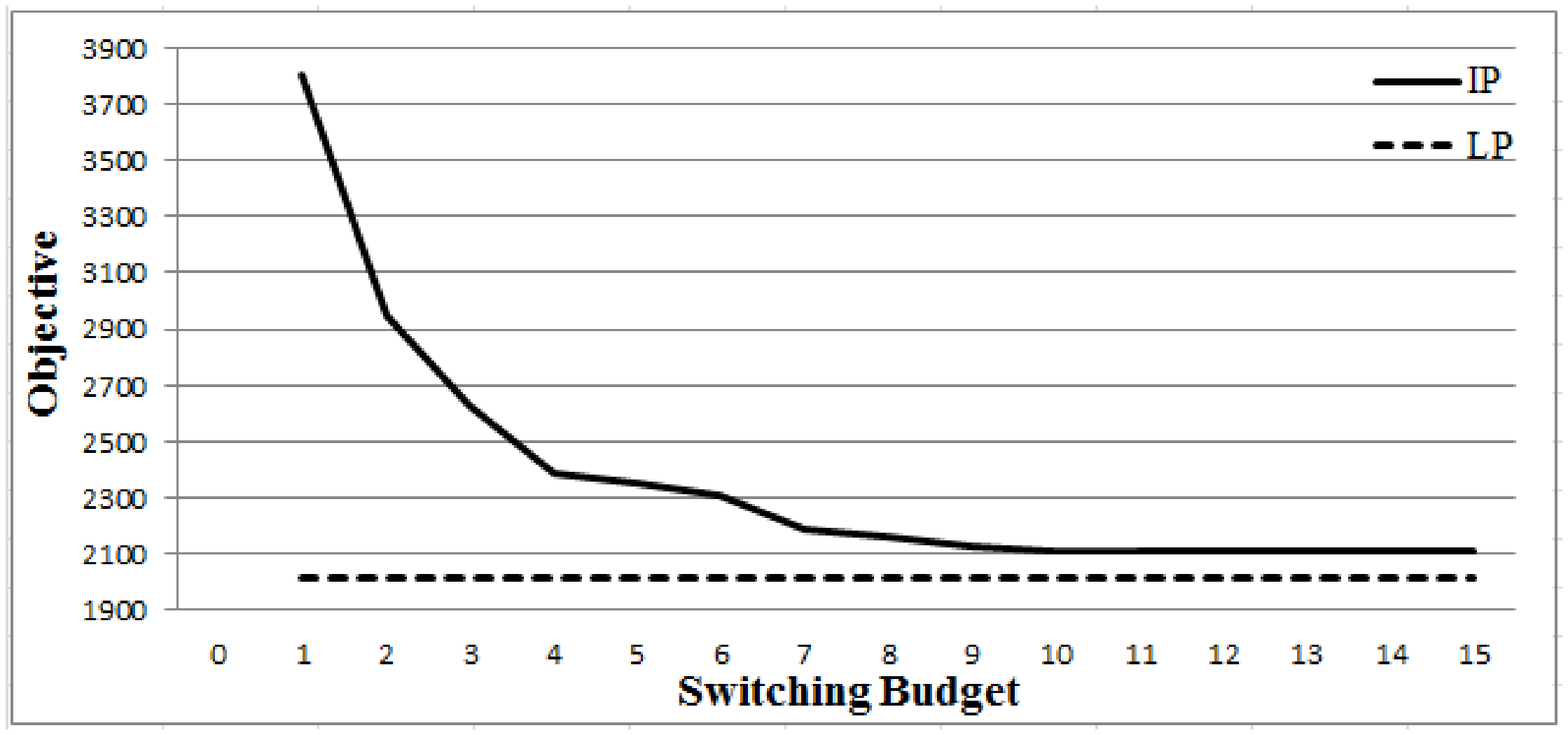}
		\caption{An instance from Set 118\_15\_16.}  
		\label{fig:card118-15-16}
	\end{subfigure}
	\vspace{3mm}
	
	\begin{subfigure}[b]{4.5in}
		\includegraphics[width=4.5in]{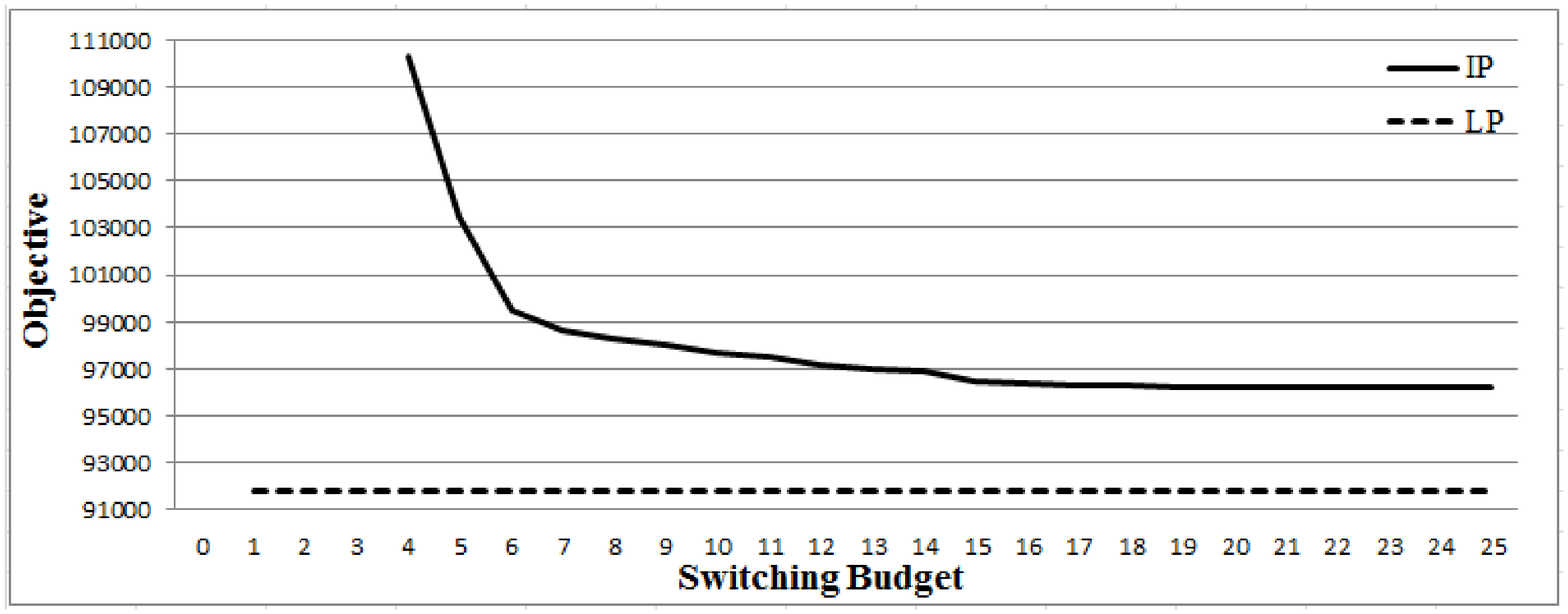}
		\caption{An instance from Set 300\_5.} 
		\label{fig:card300-5}
	\end{subfigure}
	\;
	\end{center}
	\caption{\normalsize {\bred Evolution of objective function (IP) and linear programming relaxation (LP) with respect to different switching budgets for five instances.}}
\end{figure}

}

\section{Conclusions}\label{section:conc}
In this paper, we propose new cycle-based formulations for the optimal power flow problem and the optimal transmission switching problems that use the DC approximation to the power flow equations.  We characterize the convex hull of a cycle-induced substructure in the new formulation, which provides strong valid inequalities that we may add to improve the new formulation.  We demonstrate that separating the new inequalities may be done in linear time for a fixed cycle. 
We conduct extensive experiments to show that the valid inequalities are very useful in reducing the size of the search tree and the computation time for the DC optimal transmission switching problem.

The inequalities we derive may be gainfully employed for {\em any}
power systems problem that involves the addition or removal of
transmission lines and for which the DC approximation to power flow is
sufficient for engineering purposes. We will pursue the application of
these inequalities to other important power systems planning and
operations problems as a future line of research.  Other future lines
of research include the investigation of more complicated substructure
of the new formulation, and engineering the cutting plane procedure to
effectively solve larger-scale networks.  {\bred As an example of
  studying more complicated substructures, one reviewer observed that
  the substructure we study does not involve flow balance constraints.
  Stronger relaxations could be obtained by separating cutting-planes
  using an extended formulation similar to \eqref{disjunctive} that
  includes flow balance constraints. 
}

 \section*{Acknowledgment}
{\bred 
  We would like to thank the reviewers for their constructive comments and in particular suggesting the experiments in Section 6.3. These have helped in significantly improving the paper.  
The work of authors Jeon, Linderoth, and Luedtke was supported in part by the U.S. Department of Energy, Office of Science, Office of Advanced Scientific Computing Research, Applied Mathematics program under contract number DE-AC02-06CH11357.
}

\appendix

\section{Cycle Basis Algorithm}\label{app:generatecycles}

\exclude{
Consider a directed graph $G=(V, E)$ with vertex set $V$ and arc set $A$. Let $|V|=n$ and $|E|=m$.
We define edge-node incidence matrix $A$ as 
\begin{equation}  \label{DefineA}
A_{(i,j),k} = \begin{cases}
1 & \text{if } i=k \\
-1 & \text{if } j=k \\
0 &	\text{otherwise}
\end{cases}
\end{equation}
Then, assuming that $G$ is connected, Algorithm \ref{Cycle basis generation.} can be used to find a cycle basis.
\begin{algorithm}
\caption{Cycle basis generation.}
\label{Cycle basis generation.}
\begin{algorithmic}
\STATE Define edge-node incidence matrix $A$ of directed graph $G$ as given in (\ref{DefineA}).
\STATE Carry out LU decomposition of A with partial pivoting to compute $PA = LU$ with  a unit lower triangular matrix $L$.
\STATE Last $m - n + 1$ rows of $L^{-1}P$, denoted by $C_b$, gives a cycle basis.
\end{algorithmic}
\end{algorithm}
}

\begin{prop} 
	Algorithm \ref{Cycle basis generation.} works correctly.
\end{prop}
\begin{proof}
Without loss of generality, assume that the first $n-1$ rows of $A$ are selected such that no row permutation is necessary during LU decomposition. In the remaining of the proof, we will replace $PA$ with $A$ for brevity.

The LU decomposition of $A$ can be obtained by a sequence of Gaussian eliminations on $A$ as 
\begin{align*}
\tilde A_1 =\tilde L_1 \cdot A, \;\; \tilde A_2 = \tilde L_2 \cdot \tilde A_1, \dots, \;\; U  = \tilde A_{n-1} =\tilde L_{n-1} \cdot \tilde A_{n-1},
\end{align*}
where each matrix $\tilde L_i$ is an elementary row operation that adds or subtracts multiples of the $i$-th row of $\tilde A_{i-1}$ to other rows to make the $i$-th column of $\tilde A_{i-1}$ the $i$-th unit vector.  Consider a nonzero entry $a_{i1}$ of $A$. Since $a_{11}, a_{i1}\in \{+1, -1\}$, the row operation only adds $+1$ or $-1$ copy of row $1$ to row $i$, that is, the first column of $\tilde L_1$ only contains $0, \pm 1$. Also, after eliminating $a_{i1}$, row $i$ of $A_1$ will either be all zero, or contain exactly one $1$ and one $-1$. In other words, $A_{1}$ is an arc-node incidence matrix for a new digraph $G_1=(V,E_1)$.  Since rank$(A)=n-1$, we have rank$(\tilde A_1)=n-1$, which implies the new digraph $G_1$ is connected. Repeating this argument for each subsequent round of Gaussian elimination, we have that $U$ is an incidence matrix of the connected digraph $G_{n-1}$ with $n-1$ arcs, which implies $G_{n-1}$ is a spanning tree of the node set $V$. Denote the first $n-1$ rows of $U$ as $U_1$. The last $m-n+1$ rows of $U$ are zeros. 


Denote $A = \begin{bmatrix} A_1 \\ A_2 \end{bmatrix}$ where $A_1$ is the first $n - 1$ rows of  $A$ and represents a spanning tree ${T}$ in the original graph ${G}$. Note that the rows of $A_1$ are linearly independent.
Let us first carry out the LU decomposition of  $A_1$ to get $A_1 = L_1U_1$. In fact, $U_1$  represents a spanning tree, say ${T'}$, on a new graph ${G'}=(V, E')$. Note that the entries of $L_1$ are precisely the negative of the pivots in Gaussian elimination and hence, they are $\pm 1$. Moreover, we can interpret the rows of $L_1$ indexed by the edges in ${T}$ and columns indexed by the edges in ${T'}$. In particular, the elements of row $(i,j) \in {T} $ represent  the unique path in ${T'} $ going from $i$ to $j$.

\begin{clm} \label{path transform}
	A path in ${T}$ can be mapped to a path in  ${T'}$ by post-multiplication of $L_1$ and this transformation is unique.
\end{clm}
\begin{proof}
Let us consider a path $p$ in ${T}$ as a row vector where +1 (-1) means an arc is traversed in forward (backward) direction and 0 means that arc is not part of the path. Define  $p'=p L_1$.  We claim that the row vector $p'$ is a path in ${T'}$. Let us traverse the path $p$ in terms of the edges in ${T'}$. In particular, we weight the rows of $\mathcal{T}$ corresponding to $(i,j)$ with the value of that edge in the path $p$. In other words, for each arc $(i,j)$ in the path, we traverse the path from $i$ to $j$ in $\mathcal{T'}$. But, this gives a path in $\mathcal{T'}$. Finally, this transformation is unique since the path joining two nodes in a tree is unique.
\end{proof}

Now, consider  $A' = \begin{bmatrix} U_1 \\ A_2 \end{bmatrix}$. 
We continue LU decomposition on $A'$ to obtain  $A'=\begin{bmatrix} U_1 \\ A_2 \end{bmatrix}=\begin{bmatrix} I & 0 \\ L_2 & I \end{bmatrix}\begin{bmatrix} U_1 \\ 0 \end{bmatrix}$. In particular, we have $A_2 = L_2U_1$. Since the rows of $U_1$ are linearly independent and $U_1$ defines a tree, the elements of $A_2$ can be traced via a unique path in ${T'}$. In fact, the paths are exactly $L_2$ in the new network. If the paths in $L_2$ are traced backwards, we obtain cycles in ${G'}$.
Hence, $\begin{bmatrix} -L_2 & I \end{bmatrix}$ is a cycle basis in ${G'}$. 

At this point, we can write $A=LU$  
where
$$
L =\begin{bmatrix} L_1 & 0 \\ 0 & I \end{bmatrix}\begin{bmatrix} I & 0 \\ L_2 & I \end{bmatrix} \text{ and }
L^{-1} =\begin{bmatrix} L_1^{-1} & 0  \\ -L_2 L_1^{-1} & I \end{bmatrix} 
$$

Finally, we claim that $C_b = \begin{bmatrix} -L_2L_1^{-1} & I \end{bmatrix}$ is a cycle basis in ${G}$. Let us first focus on the system $L_2 = ML_1$. Recall that the rows of $L_2$ are paths in ${G'}$. We claim that the rows of $M$ are the corresponding paths in $G$.  Using Claim \ref{path transform}, we know that post-multiplication of a path in $G$ by $L_1$ gives a path in $G'$. But, since $L_1$ is invertible,  $M=L_2L_1^{-1}$ is the unique solution and therefore, the rows of $M$ should represent paths in ${G}$. Then, by tracing the  paths in $M$  backwards, we obtain cycles in ${G}$.
Therefore, $\begin{bmatrix} -L_2L_1^{-1} & I \end{bmatrix}$ is a cycle basis in ${G}$. 
%
\end{proof}

Note that we do not need to explicitly invert $L$ to obtain $L^{-1}$. In fact, LU decomposition produces $L^{-1}$. Hence, it is computationally efficient to find cycle basis using Algorithm \ref{Cycle basis generation.}.

\bibliographystyle{ormsv080}
\bibliography{references}


\end{document}